\documentclass{amsart}
\pdfoutput=1
\usepackage[utf8]{inputenc}
\usepackage[T1]{fontenc}
\usepackage[english]{babel}
\usepackage[babel=true]{microtype}
\usepackage{amsmath,amsthm,amssymb,bookmark,longtable,booktabs}
\usepackage[backend=biber,style=alphabetic,giveninits]{biblatex}
\DeclareNameAlias{default}{family-given}
\theoremstyle{plain}
\newtheorem{theorem}{Theorem}[section]
\newtheorem{lemma}[theorem]{Lemma}
\newtheorem{corollary}[theorem]{Corollary}
\newtheorem{conjecture}[theorem]{Conjecture}
\newtheorem{proposition}[theorem]{Proposition}

\theoremstyle{definition}
\newtheorem{definition}[theorem]{Definition}
\newtheorem{example}[theorem]{Example}
\newtheorem{notation}[theorem]{Notation}

\theoremstyle{remark}
\newtheorem{remark}[theorem]{Remark}
\newtheorem*{conventions}{Notations and conventions}
\newtheorem*{acknowledgements}{Acknowledgements}
\newtheorem*{structure}{Structure of the paper}

\title[The classification of smooth well-formed Fano weighted complete intersections]{The classification of smooth well-formed \\ Fano weighted complete intersections}
\author{Mikhail Ovcharenko}
\address{Steklov Mathematical Institute of Russian Academy of Sciences, 8 Gubkina street, Moscow 119991, Russia}
\email{ovcharenko@mi-ras.ru}

\DeclareMathOperator{\Ass}{Ass}
\DeclareMathOperator{\Bs}{Bs}
\DeclareMathOperator{\Cl}{Cl}
\DeclareMathOperator{\codim}{codim}
\DeclareMathOperator{\coind}{coind}
\DeclareMathOperator{\corank}{corank}
\DeclareMathOperator{\HS}{HS}
\DeclareMathOperator{\height}{ht}
\DeclareMathOperator{\im}{im}
\DeclareMathOperator{\lcm}{lcm}
\DeclareMathOperator{\Mac}{Mac}
\DeclareMathOperator{\Mat}{Mat}
\DeclareMathOperator{\MHS}{MHS}
\DeclareMathOperator{\Pic}{Pic}
\DeclareMathOperator{\Proj}{Proj}
\DeclareMathOperator{\pr}{prime}
\DeclareMathOperator{\pt}{pt}
\DeclareMathOperator{\Sing}{Sing}
\DeclareMathOperator{\Spec}{Spec}
\DeclareMathOperator{\rad}{rad}
\DeclareMathOperator{\rank}{rank}
\DeclareMathOperator{\var}{var}

\calclayout

\begin{document}

\begin{abstract}
  We show that the set of families of smooth well-formed Fano weighted complete intersections admits a natural partition with respect to the variance \(\var(X) = \coind(X) - \codim(X)\). Moreover, we obtain the classification of smooth well-formed Fano weighted complete intersections of small variance.

  We also prove that the anticanonical linear system on a smooth well-formed Fano weighted complete intersection of anticanonical degree one is never base-point free.
\end{abstract}

\maketitle

\section{Introduction}

Fano varieties constitute a fundamental part of the classification of algebraic varieties. The number of deformation classes of smooth Fano varieties of given dimension is finite (see~\cite[Theorem~0.2]{kollar/boundness}), so that one can hope for their explicit classification. A basic invariant of a smooth Fano variety \(X\) is its \emph{index}: the largest natural number \(i_X > 0\) such that \(- K_X\) is divisible by \(i_X\) in \(\Pic(X)\).

Smooth Fano varieties of large index can be explicitly described.

\begin{theorem}[{\cite{kobayashi/characterizations}}]
  For any smooth Fano variety \(X\) of dimension \(n\) we have the estimate \(i_X \leqslant n + 1\). Moreover, the following assertions hold:
  \begin{itemize}
  \item \(i_X = n + 1\) if and only if \(X \simeq \mathbb{P}^{n}\);
  \item \(i_X = n\) if and only if \(X \subset \mathbb{P}^{n + 1}\) is a smooth quadric of dimension \(n > 1\).
  \end{itemize}
\end{theorem}

The number \(\coind(X) = \dim(X) + 1 - i_X\) is usually referred to as the \emph{coindex} of a smooth Fano variety \(X\). In these terms, the result of Kobayashi--Ochiai describes smooth Fano varieties of coindex 0 and 1, respectively. Smooth Fano varieties of coindex 2 were classified by Iskovskikh and Fujita (see~\cite[Theorem~3.3.1]{iskovskikh/fano}).

The goal of this paper is to describe the classification of smooth Fano varieties which can be realised as complete intersections in a weighted projective space. We restrict ourselves to the case of \emph{well-formed} weighted complete intersections: they are well-behaved with respect to the description in terms of weights and degrees (see Definition~\ref{definition:WF-subscheme} and Example~\ref{example:WCI-non-WF-pathologies}). Moreover, we assume that our weighted complete intersections are not \emph{intersections with linear cones} (see Definition~\ref{definition:WCI-degenerate}), which is a straightforward generalisation of the notion of a non-degenerate projective variety. The latter condition is rather mild: it can always be assumed provided that a weighted complete intersection is general enough (see~Proposition~\ref{proposition:WCI-general-QS-implies-WF}).

By definition we can associate to a weighted complete intersection \(X_{\mu} \subset \mathbb{P}(\rho)\) in the weighted projective space \(\mathbb{P}(\rho)\) the weights \(\rho = (a_0, \ldots, a_N)\) of \(\mathbb{P}(\rho)\) and the \emph{multidegree} \(\mu = (d_1, \ldots, d_c)\) of \(X\) (see Definition~\ref{definition:WCI}). Weighted complete intersections in \(\mathbb{P}(\rho)\) of given multidegree \(\mu\) form a family \(\mathcal{F}^{\rho}_{\mu}\) (in the sense of Definition~\ref{definition:WCI-family}). Throughout this paper the codimension of \(X\) is defined as \(\codim(X) = \dim(\mathbb{P}(\rho)) - \dim(X) = c\).

Any smooth well-formed Fano weighted complete intersection of dimension at least 2 which is not an intersection with a linear cone is contained in a unique family \(\mathcal{F}^{\rho}_{\mu}\) (see Proposition~\ref{proposition:WCI-QS-WF-implies-unique} and Remark~\ref{remark:WCI-QS-WF-implies-unique}). In this paper we show that the set of these families has a natural structure. The starting point here is the following result of Przyjalkowski--Shramov which can be considered as an analogue of the classification of smooth Fano varieties of minimal coindex.

\begin{theorem}[{\cite[Theorem~1.1]{przyjalkowski/codimension}}]\label{theorem:WCI-fano-small-variance}
  Let \(X \subset \mathbb{P}(\rho)\) be a smooth well-formed Fano weighted complete intersection of dimension at least 2 which is not an intersection with a linear cone.

  The following assertions hold.
\begin{enumerate}
\item One has \(\codim(X) \leqslant \coind(X)\).
\item If \(\codim(X) = \coind(X)\), then \(X\) is a complete intersection of \(\coind(X)\) quadrics in \(\mathbb{P}(\rho) = \mathbb{P}^{\dim(\mathbb{P}(\rho))}\).
\item Suppose that \(\codim(X) = \coind(X) - 1\), and \(\codim(X) > 1\). Then \(X\) is a complete intersection of \(\coind(X) - 2\) quadrics and a cubic in \(\mathbb{P}(\rho) = \mathbb{P}^{\dim(\mathbb{P}(\rho))}\).
\end{enumerate}
\end{theorem}

\begin{remark}
  Let \(X \subset \mathbb{P}(\rho)\) be a smooth well-formed Fano weighted complete intersection which is not an intersection with a linear cone. Assume that \(\codim(X) = 1\), and \(\coind(X) = 2\). Then~\cite[Corollary 3.8]{przyjalkowski/automorphisms} implies that \(X\) is a sextic hypersurface in \(\mathbb{P}(1^{(i_X + 1)}, 2, 3)\) or a quartic hypersurface in \(\mathbb{P}(1^{(i_X + 2)}, 2)\).
\end{remark}

\begin{definition}
  Let \(X \subset \mathbb{P}(\rho)\) be a smooth well-formed Fano weighted complete intersection which is not an intersection with a linear cone. We define the \emph{variance} of \(X\) as
  \[
    \var(X) =
    \begin{cases}
      \coind(X) - \codim(X) & \text{ if } \dim(X) > 1, \\
      0 & \text{ otherwise}.
    \end{cases}
  \]
\end{definition}

\begin{notation}
  We denote by \(\Theta\) the set of families \(\mathcal{F}^{\rho}_{\mu}\) of weighted complete intersections in \(\mathbb{P}(\rho)\) of multidegree~\(\mu\) (in the sense of Definition~\ref{definition:WCI-family}) such that there exists a smooth well-formed Fano weighted complete intersection \(X \in \mathcal{F}^{\rho}_{\mu}\) which is not an intersection with a linear cone.
\end{notation} 

Variance of a smooth well-formed Fano weighted complete intersection is a characteristic of the corresponding family \(\mathcal{F}^{\rho}_{\mu} \in \Theta\) of weighted complete intersections (cf. Corollary~\ref{corollary:WCI-fano-index}):
\[
  \var \colon \mathcal{F}^{\rho}_{\mu} \mapsto
  \sum_{j = 1}^c (d_j - 2) - \sum_{i = 0}^N (a_i - 1); \quad
  \rho = (a_0, \ldots, a_N), \quad \mu = (d_1, \ldots, d_c).
\]

\begin{remark}
  Instead of variance one can work with the difference \(\dim(X) - \codim(X) = \var(X) + (i_X - 1)\), where \(\var(X) \geqslant 0\) by Theorem~\ref{theorem:WCI-fano-small-variance}, and \(i_X > 0\) by definition. For example, see~\cite[Lemma~3.9]{przyjalkowski/codimension}.
\end{remark}

Our main goal is to explicitly construct a partition of the set \(\Theta\) into subsets of families of constant variance (Propositions~\ref{proposition:WCI-fano-indexes} and~\ref{proposition:WCI-fano-structure}), and bound their number in terms of variance (Theorems~\ref{theorem:WCI-fano-s2-bound} and~\ref{theorem:WCI-fano-degree-one}).

\begin{notation}
  Let \(\mathcal{F} = \mathcal{F}^{\rho}_{\mu}\) be a family of weighted complete intersections. We introduce the following families of weighted complete intersections:
  \[
    \mathcal{F}^{(l)}_{(m)} = \mathcal{F}^{\widetilde{\rho}}_{\widetilde{\mu}}, \quad
    \widetilde{\rho} = (1^{(l + 2 m)}; \rho), \quad \widetilde{\mu} = (2^{(m)}; \mu); \quad
    l, m \in \mathbb{Z}_{\geqslant 0}.
  \]
\end{notation}

We can explicitly describe when the map \(\mathcal{F} \mapsto \mathcal{F}^{(l)}_{(m)}\) is well-defined on the set \(\Theta\).

\begin{definition}\label{definition:WCI-amenable}
  We refer to a smooth well-formed weighted complete intersection \(X \subset \mathbb{P}(a_0, \ldots, a_N)\) as
  \begin{itemize}
  \item \emph{amenable} if \(X\) is non-degenerate (see Definition~\ref{definition:WCI-degenerate}), and \(\vert \mathcal{O}_X(1) \vert\) is base-point free;
  \item \emph{strictly amenable} if \(X\) is amenable, and the following conditions hold:
    \begin{itemize}
    \item \(X\) is not an intersection with a linear cone;
    \item we have \(a_i \neq 2\) for any \(i = 0, \ldots, N\).
    \end{itemize}
  \end{itemize} 
\end{definition}

\begin{remark}
  If a family \(\mathcal{F} \in \Theta\) contains a (strictly) amenable Fano weighted complete intersection, then its general element is also (strictly) amenable and Fano (see Corollary~\ref{corollary:WCI-smoothness-genericity}).
\end{remark}

\begin{example}\label{example:WCI-non-amenable}
  We have the following examples of (non-)amenable weighted complete intersections.
  \begin{enumerate}
  \item A smooth well-formed hypersurface \(X_{2 m} \subset \mathbb{P}(1^{(l + 1)}, m)\) of degree \(2 m\) is amenable for any \(l,m \in \mathbb{Z}_{> 0}\). It is a double cover of \(\mathbb{P}^l\) branched over a hypersurface \(Y_{2 m} \subset \mathbb{P}^l\).
  \item A smooth well-formed Fano weighted complete intersection \(X_{(3, 10)} \subset \mathbb{P}(1^{(i_X + 6)}, 2, 5)\) is not amenable.
  \end{enumerate}
\end{example}

\begin{remark}\label{remark:WCI-amenable-O1-freeness}
  Let \(X\) be a smooth well-formed weighted complete intersection. Since the sheaf \(\mathcal{O}_X(1)\) is ample, \(X\) is amenable if and only if the associated map \(\Phi_{\vert \mathcal{O}_X(1) \vert} \colon X \dashrightarrow \vert \mathcal{O}_X(1) \vert\) is a finite morphism (see, for example,~\cite[Proposition~1.2.13]{lazarsfeld/positivity}). Moreover, if \(X\) is non-degenerate, we can identify \(\vert \mathcal{O}_X(1) \vert\) with \(\vert \mathcal{O}_{\mathbb{P}(\rho)}(1) \vert\) by Lemma~\ref{lemma:WCI-QS-WF-O1-dimension}. Then we can interpret the map \(\Phi_{\vert \mathcal{O}_X(1) \vert}\) as the projection to \(\vert \mathcal{O}_{\mathbb{P}(\rho)}(1) \vert \subset \mathbb{P}(\rho)\). 
\end{remark}

\begin{proposition}[Corollaries~\ref{corollary:WCI-fano-upper-index} and~\ref{corollary:WCI-fano-lower-index}]\label{proposition:WCI-fano-indexes}
  The following assertions hold for any \(\mathcal{F} \in \Theta\) and \(l, m \in \mathbb{Z}_{> 0}\).
  \begin{enumerate}
  \item We always have \(\mathcal{F}^{(l)}_{(0)} \in \Theta\).
  \item \(\mathcal{F}^{(0)}_{(m)} \in \Theta\) holds if and only if \(\mathcal{F}\) contains a strictly amenable weighted complete intersection.
  \end{enumerate}
\end{proposition}

We can split the set \(\Theta\) into disjoint union of \emph{series} and \emph{semi-series}.

\begin{definition}\label{definition:series-semiseries}
  Let \(\mathcal{F} \in \Theta\) be a family of weighted complete intersections such that \(\mathcal{F}\) cannot be presented as \(\mathcal{F} = \mathcal{G}^{(l)}_{(m)}\) for all families \(\mathcal{G} \in \Theta \setminus \{\mathcal{F}\}\), and \(l, m \in \mathbb{Z}_{\geqslant 0}\).
  \begin{itemize}
  \item Assume that \(\mathcal{F}\) contains a strictly amenable Fano weighted complete intersection. The \emph{series} generated by the family \(\mathcal{F}\) is the subset
  \(
  \mathfrak{S}(\mathcal{F}) = \{\mathcal{G} \in \Theta \mid \mathcal{G} =
  \mathcal{F}^{(l)}_{(m)}; \; l,m \in \mathbb{Z}_{\geqslant 0} \} \subset \Theta.
  \)
  \item Assume that \(\mathcal{F}\) does not contain a strictly amenable Fano weighted complete intersection. The \emph{semi-series} generated by the family \(\mathcal{F}\) is the subset
  \(
  \mathfrak{S}(\mathcal{F}) = \{\mathcal{G} \in \Theta \mid \mathcal{G} =
  \mathcal{F}^{(l)}_{(0)}; \; l \in \mathbb{Z}_{\geqslant 0} \} \subset \Theta.
  \)
  \end{itemize}
\end{definition}

\begin{proposition}[Lemma~\ref{lemma:WCI-fano-generator}, Corollary~\ref{corollary:WCI-fano-structure}]\label{proposition:WCI-fano-structure}
  The following assertions hold.
  \begin{enumerate}
  \item Let \(\mathcal{F}^{\rho}_{\mu} \in \Theta\) be a family of weighted complete intersections in \(\mathbb{P}(\rho)\) of multidegree \(\mu\). Then \(\mathcal{F}^{\rho}_{\mu}\) generates a series or a semi-series if and only the following conditions hold:
  \[
    \sum_{i = 0}^N a_i - \sum_{j = 1}^c d_j = 1, \quad
    \vert \{j \mid d_j > 2\} \vert = c; \quad
    \rho = (a_0, \ldots, a_N), \quad \mu = (d_1, \ldots, d_c).
  \]
  \item Series and semi-series form the partition of the set \(\Theta\).
  \end{enumerate}
\end{proposition}

\begin{remark}\label{remark:WCI-fano-index}
  Let \(\mathcal{F}^{\rho}_{\mu} \in \Theta\) be a family of weighted complete intersection of dimension at least 2. Then a general element \(X \in \mathcal{F}^{\rho}_{\mu}\) is a smooth Fano variety of index \(i_X = \sum_{i = 0}^N a_i - \sum_{j = 1}^c d_j\) (see Corollary~\ref{corollary:WCI-fano-index}).
\end{remark}

Now we are going to bound the number of series and semi-series in terms of variance. In order to do this, we have to bound the \emph{quadratic irregularity} of a smooth well-formed Fano weighted complete intersection.

\begin{definition}
  Let \(X \subset \mathbb{P}(\rho)\) be a smooth well-formed weighted complete intersection of multidegree \((d_1, \ldots, d_c)\) which is not an intersection with a linear cone. We refer to the number \(s_2(X) = \vert \{j \mid d_j > 2 \} \vert\) as the \emph{quadratic irregularity} of \(X\). 
\end{definition}

\begin{theorem}[see Corollary~\ref{corollary:WCI-fano-s2-bound}]\label{theorem:WCI-fano-s2-bound}
  Let \(X \subset \mathbb{P}(\rho)\) be a smooth well-formed Fano weighted complete intersection which is not an intersection with a linear cone. Assume that \(\var(X) > 0\).

  The following estimate holds: \(s_2(X) \leqslant 3 \var(X) - 2\).
\end{theorem}

\begin{conjecture}
  Let \(X \subset \mathbb{P}(\rho)\) be a smooth well-formed Fano weighted complete intersection which is not an intersection with a linear cone. The following estimate holds: \(s_2(X) \leqslant \var(X)\).
\end{conjecture}

\begin{remark}
  The conjecture holds for \(\var(X) < 5\) (see Appendix~\ref{section:small-variance}).
\end{remark}

Let us formulate the main result of the paper.

\begin{notation}
  For any \(r \in \mathbb{Z}_{\geqslant 0}\) let us denote by \(\mathbb{S}_r\) the set of all series and semi-series of variance \(r\). We can split it into the following three disjoint subsets \(\mathbb{S}_r = \mathbb{S}^{\alpha}_r \bigsqcup \mathbb{S}^{\beta}_r \bigsqcup \mathbb{S}^{\gamma}_r\):
  \begin{itemize}
  \item \(\mathbb{S}^{\alpha}_r\) consists of series of families of complete intersections in \(\mathbb{P}(\rho) = \mathbb{P}^{\dim(\mathbb{P}(\rho))}\);
  \item \(\mathbb{S}^{\beta}_r\) consists of series of families of weighted complete intersections in \(\mathbb{P}(\rho) \neq \mathbb{P}^{\dim(\mathbb{P}(\rho))}\).
  \item \(\mathbb{S}^{\gamma}_r\) consists of semi-series of families of weighted complete intersections in \(\mathbb{P}(\rho) \neq \mathbb{P}^{\dim(\mathbb{P}(\rho))}\).
  \end{itemize}

  In other words, the following identity holds:
  \[
    \Theta = \bigsqcup_{r = 0}^{\infty}
    \left (
      \bigsqcup_{\mathfrak{S} \in \mathbb{S}_r} \mathfrak{S}
    \right ) = \bigsqcup_{r = 0}^{\infty}
    \left (
      \bigsqcup_{\mathfrak{S}_{\alpha} \in \mathbb{S}^{\alpha}_r} \mathfrak{S}_{\alpha}
      \sqcup
      \bigsqcup_{\mathfrak{S}_{\beta} \in \mathbb{S}^{\beta}_r} \mathfrak{S}_{\beta}
      \sqcup
      \bigsqcup_{\mathfrak{S}_{\gamma} \in \mathbb{S}^{\gamma}_r} \mathfrak{S}_{\gamma}
    \right ).
  \]
\end{notation}

\begin{theorem}[Lemma~\ref{lemma:CI-fano-series-cardinality}, Corollary~\ref{corollary:WCI-fano-series-semiseries-bound}]\label{theorem:WCI-fano-series-semiseries-bound}
  The following assertions hold.
  \begin{enumerate}
  \item \(\vert \mathbb{S}^{\alpha}_r \vert\) is equal to the number of integer partitions of \(r\) for any \(r \in \mathbb{Z}_{\geqslant 0}\).
  \item Put \(F_w(m) = \binom{m + w - 1}{m}\). The following estimates hold:
    \[
      \vert \mathbb{S}^{\beta}_r \vert \leqslant \sum_{c = 1}^{3 r - 2} F_r(c - 1) F_r(c), \quad
      \vert \mathbb{S}^{\gamma}_r \vert \leqslant
      \sum_{c = 1}^{3 r - 2} F_{r + 2 c}(r + c) F_c((r - 1) (r + 2 c)).
    \]
  \end{enumerate}
\end{theorem}

\begin{remark}
  The estimates provided by Theorem~\ref{theorem:WCI-fano-series-semiseries-bound} are far from being strict (cf. Appendix~\ref{section:small-variance}).
\end{remark}

Smooth well-formed Fano weighted complete intersections play an important role in the study of birational rigidity of Fano varieties (see~\cites{cheltsov/rigid,pukhlikov/rigid}). It is well-known that a smooth Fano variety \(X\) of dimension at least 3 such that \((-K_X)^{\dim(X)} \leqslant 4\), and \(\Bs(\vert -K_X \vert) = \varnothing\), is birationally superrigid (see~\cite[\textsection0.3]{cheltsov/rigid}). At the same time, there are few Fano varieties satisfying these two conditions (for example, double spaces and double quadrics, see~\cite[Example~2.1.3]{pukhlikov/rigid}). We prove that the linear system \(\vert -K_X \vert\) on a smooth well-formed Fano weighted complete intersection \(X\) of anticanonical degree one is \emph{never} base-point free.

\begin{theorem}[Proposition~\ref{proposition:WCI-fano-degree-one}]\label{theorem:WCI-fano-degree-one}
  Let \(X \subset \mathbb{P}(\rho)\) be a smooth well-formed Fano weighted complete intersection of multidegree \((d_1, \ldots, d_c)\) which is not an intersection with a linear cone. Assume that \((-K_X)^{\dim(X)} = 1\).
  \begin{enumerate}
  \item The following estimate holds: \(\dim(\vert \mathcal{O}_X(1) \vert) < \dim(X)\), hence \(\Bs(\vert \mathcal{O}_X(1) \vert) \neq \varnothing\) (see Remark~\ref{remark:WCI-amenable-O1-freeness}).
  \item The degree \(d_j\) is not a prime power for any \(j = 1, \ldots, c\).
  \end{enumerate}
\end{theorem}

\begin{corollary}[Corollary~\ref{corollary:WCI-fano-degree-one-finite}]
  The number of families \(\mathcal{F} \in \Theta\) such that there exists a smooth well-formed Fano weighted complete intersection \(X \in \mathcal{F}\) of anticanonical degree one and given variance is finite. 
\end{corollary}

\begin{structure}
  In Section~\ref{section:preliminaries} we provide various preliminary results on weighted complete intersections. In Section~\ref{definition:series-semiseries} we prove main results of the paper:
\begin{itemize}
\item in Subsection~\ref{subsection:structure} we explicitly construct the partition of the set \(\Theta\) into series and semi-series of smooth well-formed Fano weighted complete intersections (Propositions~\ref{proposition:WCI-fano-indexes} and~\ref{proposition:WCI-fano-structure});
\item in Subsection~\ref{subsection:s2-bounds} we bound the quadratic irregularity of a smooth well-formed Fano weighted complete intersection in terms of variance (Theorem~\ref{theorem:WCI-fano-s2-bound});
\item in Subsection~\ref{subsection:series-semiseries-bounds} we bound the number of series and semi-series of smooth well-formed Fano weighted complete intersection in terms of variance (Theorem~\ref{theorem:WCI-fano-series-semiseries-bound}).
\end{itemize}
In Section~\ref{section:degree-one} we prove that the anticanonical linear system on a smooth well-formed Fano weighted complete intersection of anticanonical degree one is never base-point free (Theorem~\ref{theorem:WCI-fano-degree-one}). Appendix~\ref{section:regular-sequences} contains statements from commutative algebra which are used in Subsections~\ref{subsection:WCI-basics} and~\ref{subsection:structure}. Appendix~\ref{section:small-variance} provides the list of series and semi-series of smooth well-formed Fano weighted complete intersections of variance up to 4.
\end{structure}

\begin{conventions}
  We work over an algebraically closed field \(\Bbbk\) of characteristic zero. Throughout the paper \(a^{(t)}\) stands for a number \(a \in \mathbb{Z}_{> 0}\) repeated \(t\) times. We assume that the weights \((a_0, \ldots, a_N)\) and degrees \((d_1, \ldots, d_c)\) of a weighted complete intersection are non-decreasing: \(a_0 \leqslant \ldots \leqslant a_N\), and \(d_1 \leqslant \ldots \leqslant d_c\).
\end{conventions}

\begin{acknowledgements}
  This work was supported by the Russian Science Foundation under grant no.~19--11--00164, \url{https://rscf.ru/en/project/19-11-00164/}. The author is grateful to V.~Przyjalkowski for stating the problem, useful discussions, and careful reading of the paper, and C.~Shramov for helpful comments and suggestions. We also want to thank the referees for their useful remarks.
\end{acknowledgements}

\section{Preliminaries}\label{section:preliminaries}

In this section we review basic facts about weighted complete intersections and their combinatorics.

\subsection{Weighted complete intersections}\label{subsection:WCI-basics}

\begin{notation}\label{notation:polynomial-ring}
  Let \(\Bbbk\) be a field, and \(\rho = (a_0, \ldots, a_N)\) be a tuple of positive integers. We denote by \(R^{\rho} = \Bbbk[X_0, \ldots, X_N]\) the polynomial ring over \(\Bbbk\) with the following structure of a graded ring:
  \[
    \deg \left ( \prod_{i = 0}^N X_i^{\alpha_i} \right ) = \sum_{i = 0}^N a_i \alpha_i, \quad
    (\alpha_0, \ldots, \alpha_N) \in \mathbb{Z}_{\geqslant 0}^{N + 1}, \quad
    R^{\rho} = \bigoplus_{n = 0}^{\infty} R^{\rho}_n.
  \]
\end{notation}

\begin{definition}
  Let \(\Bbbk\) be a field, and \(\rho = (a_0, \ldots, a_N)\) be a tuple of positive integers. We refer to \(\mathbb{P}(\rho) = \Proj(R^{\rho})\) as the \emph{weighted projective space over \(\Bbbk\) with weights \(\rho\)}.
\end{definition}

\begin{definition}[{\cite[Definition~5.11]{ianofletcher/weighted}}]
  A weighted projective space \(\mathbb{P}(a_0, \ldots, a_N)\) is said to be \emph{well-formed} if
  \(
  \gcd(a_0, \ldots, a_{i - 1}, \widehat{a_i}, a_{i + 1}, \ldots, a_N) = 1
  \)
  for an \(i = 0, \ldots, N\).
\end{definition}

\begin{proposition}[{\cite[\nopp 1.3.1]{dolgachev/weighted}}]
  Any weighted projective space is isomorphic to a well-formed one.
\end{proposition}

The singular locus of a weighed projective space can be explicitly described.

\begin{lemma}[{\cite[\nopp 5.15]{ianofletcher/weighted}}]\label{lemma:WPS-WF-singular-locus}
  Let \(\mathbb{P}(\rho) = \Proj(R^{\rho})\) be a well-formed weighted projective space with weights \(\rho = (a_0, \ldots, a_N)\). Then \(\Sing(\mathbb{P}(\rho))\) is the union of strata
  \[
    \Sing(\mathbb{P}(\rho)) = \bigcup_I \{X_i = 0 \mid i \not \in I\}, \quad
    R^{\rho} = \Bbbk[X_0 ,\ldots, X_N],
  \]
  over all subsets \(I\) of \(\{0, \ldots, N\}\) such that \(\gcd(\{a_i \mid i \in I\}) > 1\).
\end{lemma}

\begin{lemma}[{\cite[Theorem~1.4.1]{dolgachev/weighted}}]\label{lemma:WPS-WF-On-sheaves}
  Let \(\mathbb{P}(\rho) = \Proj(R^{\rho})\) be a well-formed weighted projective space. Then we can identify \(H^0(\mathbb{P}(\rho), \mathcal{O}_{\mathbb{P}(\rho)}(n)) \simeq R^{\rho}_n\) for any \(n \in \mathbb{Z}_{\geqslant 0}\).
\end{lemma}

\begin{corollary}\label{corollary:WPS-WF-singular-locus-O1-base-locus}
  Let \(\mathbb{P}(\rho)\) be a well-formed weighted projective space. Then we have \(\Sing(\mathbb{P}(\rho)) \subset \Bs(\vert \mathcal{O}_{\mathbb{P}(\rho)}(1) \vert)\).
\end{corollary}

Let us remind the correspondence between homogeneous ideals in \(R^{\rho}\) and closed subschemes in \(\mathbb{P}(\rho)\).

\begin{definition}[see~{\cite[Proposition~3.1.2]{dolgachev/weighted}}]\label{definition:defining-ideal}
  Let \(X \subset \mathbb{P}(\rho) = \Proj(R^{\rho})\) be a closed subscheme, and
  \[
    S_X \colon R^{\rho} \rightarrow \bigoplus_{n = 0}^{\infty} H^0(X, \mathcal{O}_X(n))
  \]
  be the corresponding Serre homomorphism. We refer to the ideal \(I_X = \ker(S_X)\) as the \emph{defining ideal} of \(X\).
\end{definition}

\begin{definition}\label{definition:saturation}
  Let \(I \subset R^{\rho}\) be a homogeneous ideal. We define its \emph{saturation} \(I^{\infty}\) as the defining ideal \(I_X\) of the closed subscheme
  \(
  X = \Proj(R^{\rho} / I) \subset \Proj(R^{\rho}) = \mathbb{P}({\rho})
  \). 
\end{definition}

\begin{lemma}\label{lemma:saturation-formula}
  Let \(I \subset R^{\rho}\) be a homogeneous ideal. Its saturation \(I^{\infty}\) can be explicitly described as
  \[
    I^{\infty} = \langle \{x \in R^{\rho} \mid (R^{\rho} \cdot x)_{n \cdot \lcm(\rho)} \subset
    I \text{ for all } n \gg 0\} \rangle.
  \]
\end{lemma}

\begin{proof}
  Actually, let us consider the truncation \(\widetilde{R^{\rho}} = \oplus_{n = 0}^{\infty} R^{\rho}_{n \cdot \lcm(\rho)}\) of the graded ring \(R^{\rho} = \oplus_{n = 0}^{\infty} R^{\rho}_n\). It is a graded ring which is finitely generated in degree one (after the obvious change of the grading).

  We have the canonical isomorphism \(\Proj(R^{\rho}) = \Proj(\widetilde{R^{\rho}})\) (see~\cite[2.4.7(i)]{grothendieck/elements}). Put \(\widetilde{I} = I \cap \widetilde{R^{\rho}}\), and \(X = \Proj(R^{\rho} / I) = \Proj(\widetilde{R^{\rho}} / \widetilde{I})\). By~\cite[Exercise~II.5.10]{hartshorne/geometry} the closed subscheme \(X \subset \Proj(\widetilde{R^{\rho}})\) canonically corresponds to the ideal \((\widetilde{I} : \widetilde{B}^{\infty}) \subset \widetilde{R^{\rho}}\), where \(\widetilde{B} \subset \widetilde{R^{\rho}}\) is the irrelevant ideal. More precisely, we have
  \[
    (\widetilde{I} : \widetilde{B}^{\infty}) = \ker
    \left (
      \widetilde{S}_X \colon \widetilde{R^{\rho}} \rightarrow
      \bigoplus_{n = 0}^{\infty} H^0(X, \mathcal{O}_X(n))
    \right ).
  \]
  Then we can recover \(I^{\infty}\) from the relation \(\im(S_X) = \im(\widetilde{S}_X)\), i.e., \(R^{\rho} / I^{\infty} \simeq \widetilde{R^{\rho}} / (\widetilde{I} : \widetilde{B}^{\infty})\).
\end{proof}

\begin{notation}\label{notation:polynomial-sequences}
  Let \(\rho\) and \(\mu = (d_1, \ldots, d_m)\) be tuples of positive integers. We denote by \(R^{\rho}_{\mu} = \prod_{j = 1}^m R^{\rho}_{d_j}\) the corresponding space of sequences of weighted homogeneous polynomials. We also denote by  \(\langle F \rangle = \langle F_1, \ldots, F_m \rangle\) the homogeneous ideal in \(R^{\rho}\) generated by a sequence \(F = (F_1, \ldots, F_m) \in R^{\rho}_{\mu}\).
\end{notation}

\begin{definition}\label{definition:WCI}
  We refer to a closed subscheme \(X \subset \mathbb{P}(\rho)\) as a \emph{weighted complete intersection} of multidegree \(\mu = (d_1, \ldots, d_c) \in \mathbb{Z}_{> 0}^c\) if its defining ideal \(I_X \subset R^{\rho}\) is generated by a regular sequence \(F \in R^{\rho}_{\mu}\).
\end{definition}

The saturatedness assumption in Definition~\ref{definition:WCI} cannot be omitted, but it often holds automatically.

\begin{example}[{\cite{przyjalkowski/weighted}}]
  Put \(\rho = (1, 1, 3, 6)\), and let \(I = \langle x_0^3, x_1^3 \rangle \subset R^{\rho}\) be a homogeneous ideal. Its saturation \(I^{\infty} = \langle x_0^3, x_1^3, x_0^2 x_1^2 \rangle\) cannot be generated by a regular sequence.
\end{example}

\begin{example}
  Let \(I \subset R^{\rho}\) be a homogeneous ideal generated by a regular sequence. Assume that \(\rho = (1^{(N + 1)})\), and \(\height(I) \leqslant N\). We claim that \(I\) is saturated in the sense of Definition~\ref{definition:saturation}. In our case the saturation \(I^{\infty}\) coincides with the ideal \((I : B^{\infty})\), where \(B \subset R^{\rho}\) is the irrelevant ideal (see Lemma~\ref{lemma:saturation-formula}). By~\cite[Corollary~18.14]{eisenbud/algebra} any associated prime \(\mathfrak{P} \in \Ass(I)\) is minimal over \(I\) and has height \(\height(\mathfrak{P}) = \height(I)\). Let \(I = \cap_{r = 1}^s \mathfrak{P}_r\) be the minimal primary decomposition. If \(I\) is not saturated, then there exists an element \(f \in (\mathfrak{P}_r : B^{\infty}) \setminus \mathfrak{P}_r\) for some \(r = 1, \ldots, s\) such that \(f \cdot B^m \subset \mathfrak{P}_r\) for some \(m \in \mathbb{Z}_{> 0}\). But \(f \not \in \mathfrak{P}_r\) by assumption, hence we have \(B \subset \rad(\mathfrak{P}_r)\). We obtain a contradiction:
  \(
    N + 1 = \height(B) \leqslant \height(\rad(\mathfrak{P}_r)) =
    \height(\mathfrak{P}_r) = \height(I)
  \).
\end{example}

\begin{example}\label{example:radical-ideal-saturated}
  A radical homogeneous ideal \(I \subset R^{\rho}\) is saturated (see Lemma~\ref{lemma:saturation-formula}).
\end{example}

Weighted complete intersections in \(\mathbb{P}(\rho)\) of given multidegree \(\mu\) form a family in the following sense.

\begin{definition}\label{definition:WCI-family}
  Let \(\mathbb{P}(\rho) = \Proj(R^{\rho})\) be a well-formed weighted projective space, and \(\mu = (d_1, \ldots, d_c) \in \mathbb{Z}_{> 0}^c\) be any tuple. We denote by \(\Gamma^{\rho}_{\mu} \subset \mathbb{P}(\rho) \times \prod_{j = 1}^c \vert \mathcal{O}_{\mathbb{P}(\rho)}(d_j) \vert\) the corresponding incidence variety. Put
  \begin{gather*}
    \Phi^{\rho}_{\mu} \colon R^{\rho}_{\mu} \xrightarrow{\sim}
    \prod_{j = 1}^c H^0(\mathbb{P}(\rho),
    \mathcal{O}_{\mathbb{P}(\rho)}(d_j)) \twoheadrightarrow
    \prod_{j = 1}^c \vert \mathcal{O}_{\mathbb{P}(\rho)}(d_j) \vert, \\
    U^{\rho}_{\mu} = \{F \in R^{\rho}_{\mu} \mid
    F \text{ is a regular sequence}\}, \;
    V^{\rho}_{\mu} = \{F \in R^{\rho}_{\mu} \mid
    \langle F \rangle \text{ is saturated}\}.
  \end{gather*}
  Let \(\pi^{\rho}_{\mu} \colon \Gamma^{\rho}_{\mu} \rightarrow \prod_{j = 1}^c \vert \mathcal{O}_{\mathbb{P}(\rho)}(d_j) \vert\) be the natural projection, and \(\mathcal{F}^{\rho}_{\mu} \colon T \rightarrow B\) be its corestriction to the subspace \(B = \Phi^{\rho}_{\mu}(U^{\rho}_{\mu} \cap V^{\rho}_{\mu})\). We refer to \(\mathcal{F}^{\rho}_{\mu}\) as the \emph{family of weighted complete intersections in \(\mathbb{P}(\rho)\) of multidegree \(\mu\)}.
\end{definition}

\begin{remark}\label{remark:regular-sequences-generic}
  The subset \(U^{\rho}_{\mu} \subset R^{\rho}_{\mu}\) is Zariski-open by Corollary~\ref{corollary:regular-sequences-generic}. 
\end{remark}

If the number of weights equal to 1 of a weighted projective space \(\mathbb{P}(\rho)\) is sufficiently large, then a family of weighted complete intersections in \(\mathbb{P}(\rho)\) is non-empty, and its general member is reduced.

\begin{lemma}[cf. Lemma~\ref{lemma:WCI-QS-WF-O1-dimension} and Proposition~\ref{proposition:WCI-fano-CY-O1-dim-bound}]\label{lemma:WCI-family-nonempty}
  Let \(\rho = (a_0, \ldots, a_N)\) and \(\mu = (d_1, \ldots, d_c)\) be tuples of positive integers. Assume that \(\vert \{i \mid a_i = 1\} \vert > c\). Then the corresponding family \(\mathcal{F}^{\rho}_{\mu} \colon T \rightarrow B\) of weighted complete intersections is non-empty, the defining ideal \(I_X \subset R^{\rho}\) of its general member \(X \in \mathcal{F}^{\rho}_{\mu}\) is radical, and the base \(B\) is an irreducible topological space.
\end{lemma}

\begin{proof}
  Here we use the notation of Definition~\ref{definition:WCI-family}. Let \((i_1, \ldots, i_c) \in \{0, \ldots, N\}\) be any non-repeating numbers such that \(a_{i_j} = 1\) for any \(j = 1, \ldots, c\). Then there exists a regular sequence \((X_{i_1}^{d_1}, \ldots, X_{i_c}^{d_c}) \in R^{\rho}_{\mu}\), where \(R^{\rho} = \Bbbk[X_0, \ldots, X_N]\). Consequently, the subset \(U^{\rho}_{\mu} \subset R^{\rho}_{\mu}\) of regular sequences is non-empty and open by Corollary~\ref{corollary:regular-sequences-generic}. Moreover,~\cite[Theorem~18.15]{eisenbud/algebra} implies that for any sequence \(F = (f_1, \ldots, f_c) \in U^{\rho}_{\mu}\) the corresponding ideal \(\langle F \rangle \subset R^{\rho}_{\mu}\) is radical if and only if \(F\) is contained in \(U^{\rho}_{\mu} \cap S^{\rho}_{\mu}\), where
  \[
    S^{\rho}_{\mu} = \{F \in R^{\rho}_{\mu} \mid \height(\langle F; \partial f_j/\partial x_i \rangle) > c\}.
  \]
  Note that by assumption \(S^{\rho}_{\mu}\) is not empty: for example, it contains a general element of \(R^{\widetilde{\rho}}_{\mu} \subset R^{\rho}_{\mu}\), where \(\widetilde{\rho} = (a_i \mid a_i = 1)\). Moreover, we claim that the subset \(S^{\rho}_{\mu} \subset R^{\rho}_{\mu}\) is Zariski-open. Let us consider the inclusion
  \[
    i \colon R^{\rho}_{\mu} \hookrightarrow R^{\rho}_{\mu} \times \prod_{i,j} R^{\rho}_{d_j - a_i}, \quad
    (f_1, \ldots, f_c) \mapsto (f_1, \ldots, f_c; \partial f_j / \partial x_i).
  \]
  The following subset is Zariski-open by Corollary~\ref{corollary:height-lower-bound-generic}:
  \[
    \widetilde{S}^{\rho}_{\mu} =
    \left \{
      G \in R^{\rho}_{\mu} \times \prod_{i,j} R^{\rho}_{d_j - a_i} \mid \height(\langle G \rangle) > c
    \right \}.
  \]
  Then the subset \(S^{\rho}_{\mu} = i(R^{\rho}_{\mu}) \cap \widetilde{S}^{\rho}_{\mu}\) is also Zariski-open. But any radical ideal is saturated by Example~\ref{example:radical-ideal-saturated}. Consequently, \(V^{\rho}_{\mu}\) contains a non-empty open subset, so the base \(B\) is irreducible.
\end{proof}

Let us recall basic geometrical properties of weighted complete intersections.

\begin{definition}[{\cite[Definition~6.3]{ianofletcher/weighted}}]\label{definition:WPV-quasismooth}
  A closed subscheme \(X \subset \mathbb{P}(\rho)\) is said to be \emph{quasi-smooth} if its affine cone \(\Spec(R^{\rho} / I_X)\) is smooth outside the origin, where \(I_X \subset R^{\rho}\) is the defining ideal (see Definition~\ref{definition:defining-ideal}).
\end{definition}

We have an analogue of Lemma~\ref{lemma:WCI-family-nonempty} for quasi-smooth weighted complete intersections.

\begin{lemma}\label{lemma:WCI-family-irreducible-base}
  Let \(\rho\) and \(\mu\) be tuples of positive integers. Assume that the corresponding family \(\mathcal{F}^{\rho}_{\mu} \colon T \rightarrow B\) of weighted complete intersections contains a quasi-smooth element. Then a general element \(X \in \mathcal{F}^{\rho}_{\mu}\) is quasi-smooth, its defining ideal \(I_X \subset R^{\rho}\) is radical, and the base \(B\) is an irreducible topological space.
\end{lemma}

\begin{proof}
  Here we use the notation of Definition~\ref{definition:WCI-family}. Let us consider the Zariski-open subset
  \[
    W^{\rho}_{\mu} = \{F \in R^{\rho}_{\mu} \mid \Spec(R^{\rho} / \langle F \rangle)
    \subset \mathbb{A}^{N + 1} \text{ is smooth outside the origin}\}.
  \]
  Recall that \(U^{\rho}_{\mu} \subset R^{\rho}_{\mu}\) is Zariski-open by Remark~\ref{remark:regular-sequences-generic}. By assumption \(U^{\rho}_{\mu} \cap W^{\rho}_{\mu}\) is non-empty. For any sequence \(F \in U^{\rho}_{\mu} \cap W^{\rho}_{\mu}\) the corresponding ideal \(\langle F \rangle \subset R^{\rho}_{\mu}\) is radical by~\cite[Exercise~18.9]{eisenbud/algebra}, hence is saturated by Example~\ref{example:radical-ideal-saturated}. Consequently, \(V^{\rho}_{\mu}\) contains a non-empty open subset, so the base \(B\) is irreducible.
\end{proof}

\begin{definition}[cf.~{\cite[Definition~1.1]{dimca/singularities}}]\label{definition:WF-subscheme}
  A closed subscheme \(X \subset \mathbb{P}(\rho)\) is said to be \emph{well-formed} if \(\mathbb{P}(\rho)\) is well-formed, and \(\codim_X (X \cap \Sing(\mathbb{P}(\rho))) \geqslant 2\).
\end{definition}

\begin{proposition}[{\cite[Proposition~8]{dolgachev/weighted}}]
  Let \(X \subset \mathbb{P}(\rho)\) be a quasi-smooth well-formed weighted complete intersection. Then the singular locus of \(X\) is the intersection of \(X\) with the singular locus of \(\mathbb{P}(\rho)\).
\end{proposition}

\begin{proposition}[{\cite[Corollary~2.14]{przyjalkowski/bounds}}]
  Let \(X \subset \mathbb{P}(\rho)\) be a smooth well-formed weighted complete intersection. Then \(X\) is quasi-smooth.
\end{proposition}

\begin{corollary}\label{corollary:WCI-smooth-WF-criterion}
  Let \(X \subset \mathbb{P}(\rho)\) be a weighted complete intersection. The following assertions are equivalent:
  \begin{itemize}
  \item \(X\) is smooth and well-formed;
  \item \(\mathbb{P}(\rho)\) is well-formed, \(X\) is quasi-smooth, and \(X \cap \Sing(\mathbb{P}(\rho)) = \varnothing\).
  \end{itemize}
\end{corollary}

\begin{example}\label{example:WCI-non-WF-pathologies}
  If a weighted complete intersection is smooth but ill-formed, various pathologies can arise.
  \begin{itemize}
  \item There exists a K3-surface which can be realised as a smooth and quasi-smooth hypersurface \(X\) of degree 9 in \(\mathbb{P}(1, 2, 2, 3)\) which is not well-formed (see~{\cite[6.15(ii)]{ianofletcher/weighted}}). The usual adjunction formula (see Proposition~\ref{proposition:WCI-QS-WF-adjunction}) would formally imply that \(\omega_X \simeq \mathcal{O}_X(1)\), which is nonsense.
  \item A general weighted hypersurface of multidegree \(6\) in \(\mathbb{P}(2, 3, 5^{(t)})\) is smooth but not well-formed or quasi-smooth for any \(t > 0\) (see~\cite[Example~2.9]{przyjalkowski/on-automorphisms}).
  \end{itemize}
\end{example}

\begin{definition}[{\cite[Definition~6.5]{ianofletcher/weighted}}]\label{definition:WCI-degenerate}
  Let \(X \subset \mathbb{P}(a_0, \ldots, a_N)\) be a weighted complete intersection of multidegree \((d_1, \ldots, d_c)\). We refer to \(X\) as
  \begin{itemize}
  \item \emph{an intersection with a linear cone} if we have \(a_i = d_j\) for some \(i = 0, \ldots, N\) and \(j = 1, \ldots, c\);
  \item \emph{degenerate} if we have \(d_j = 1\) for some \(j = 1, \ldots, c\).
  \end{itemize}
\end{definition}

\begin{lemma}[{\cite[Corollary~3.3]{przyjalkowski/automorphisms}}]\label{lemma:WCI-QS-WF-projective-normality}
  Let \(X \subset \mathbb{P}(\rho)\) be a quasi-smooth well-formed weighted complete intersection. Then the restriction map
  \(
    H^0(\mathbb{P}(\rho), \mathcal{O}_{\mathbb{P}(\rho)}(m)) \rightarrow
    H^0(X, \mathcal{O}_X(m))
  \)
  is surjective for any \(m \in \mathbb{Z}_{\geqslant 0}\).
\end{lemma}

\begin{lemma}[{\cite[Theorem~3.4.4]{dolgachev/weighted}},~{\cite[Lemma~7.1]{ianofletcher/weighted}}]\label{lemma:WCI-QS-WF-O1-dimension}
  Let \(X \subset \mathbb{P}(\rho)\) be a quasi-smooth well-formed weighted complete intersection which is non-degenerate. Put \(\rho = (a_0, \ldots, a_N)\). The following identities hold:
  \[
    \dim(H^0(X, \mathcal{O}_X(1))) =
    \dim(H^0(\mathbb{P}(\rho), \mathcal{O}_{\mathbb{P}(\rho)}(1))) =
    \vert \{i \mid a_i = 1\} \vert.
  \]
\end{lemma}

\begin{corollary}\label{corollary:WCI-QS-O1-base-locus}
  Let \(X \subset \mathbb{P}(\rho)\) be a quasi-smooth weighted complete intersection.
  \begin{enumerate}
  \item If \(X\) is well-formed and non-degenerate, then we have \(\Bs(\vert \mathcal{O}_X(1) \vert) = X \cap \Bs(\vert \mathcal{O}_{\mathbb{P}(\rho)}(1) \vert)\).
  \item If \(\mathbb{P}(\rho)\) is well-formed, and \(X \cap \Bs(\vert \mathcal{O}_{\mathbb{P}(\rho)}(1) \vert) = \varnothing\), then \(X\) is smooth and well-formed.  
  \end{enumerate}
\end{corollary}

\begin{proof}
  Firstly, let \(X\) be quasi-smooth, well-formed, and non-degenerate. Then Lemmas~\ref{lemma:WCI-QS-WF-projective-normality} and~\ref{lemma:WCI-QS-WF-O1-dimension} imply that the restriction map \(H^0(\mathbb{P}(\rho), \mathcal{O}_{\mathbb{P}(\rho)}(1)) \rightarrow 
  H^0(X, \mathcal{O}_X(1))\) is an isomorphism, and we are done. Secondly, let \(\mathbb{P}(\rho)\) be well-formed, and \(X\) be quasi-smooth. Then by Corollary~\ref{corollary:WCI-smooth-WF-criterion} \(X\) is smooth and well-formed if and only if \(X \cap \Sing(\mathbb{P}(\rho)) = \varnothing\). But we have \(\Sing(\mathbb{P}(\rho)) \subset \Bs(\vert \mathcal{O}_{\mathbb{P}(\rho)}(1) \vert)\) by Corollary~\ref{corollary:WPS-WF-singular-locus-O1-base-locus}.
\end{proof}

\begin{corollary}\label{corollary:WCI-smoothness-genericity}
  Let \(\mathbb{P}(\rho)\) be a well-formed weighted projective space, \(\mu\) be a tuple of positive integers, and \(\mathcal{F}^{\rho}_{\mu} \colon T \rightarrow B\) be the family of weighted complete intersections in \(\mathbb{P}(\rho)\) of multidegree \(\mu\) (see Definition~\ref{definition:WCI-family}). The following subsets of \(B\) are open:
  \[
    B' = \{X \in B \mid X \text{ is smooth and well-formed}\}, \quad
    B'' = \{X \in B \mid X \text{ is amenable}\}.
  \]
\end{corollary}

\begin{proof}
  Corollary~\ref{corollary:WCI-smooth-WF-criterion} implies that \(B'\) consists of weighted complete intersections \(X\) such that \(X\) is quasi-smooth, and \(X \cap \Sing(\mathbb{P}(\rho)) = \varnothing\) (see~Lemma~\ref{lemma:WPS-WF-singular-locus}). The first condition is open by Definition~\ref{definition:WPV-quasismooth}. The second condition is open by Lemma~\ref{lemma:hilbert-series-semicontinuity}.

  By Definition~\ref{definition:WCI-amenable} the subset \(B''\) consists of weighted complete intersections \(X \in B'\) such that \(X\) is non-degenerate, and \(\Bs(\vert \mathcal{O}_X(1) \vert) = \varnothing\). Corollary~\ref{corollary:WCI-QS-O1-base-locus} implies that we can write the latter condition as \(X \cap \Bs(\vert \mathcal{O}_{\mathbb{P}(\rho)}(1) \vert) = \varnothing\) (see Lemma~\ref{lemma:WPS-WF-On-sheaves}). Then \(B''\) is also open by Lemma~\ref{lemma:hilbert-series-semicontinuity}.
\end{proof}

\begin{proposition}[{\cite[Proposition~2.9]{przyjalkowski/codimension}}]\label{proposition:WCI-general-QS-implies-WF}
  Let \(X \subset \mathbb{P}(\rho)\) be a quasi-smooth weighted complete intersection of multidegree \(\mu\). Assume that \(\dim(X) \geqslant 3\), and \(X\) is general in the family \(\mathcal{F}^{\rho}_{\mu}\) of weighted complete intersections. Then there exists a quasi-smooth well-formed weighted complete intersection \(X'\) isomorphic to \(X\) which is not an intersection with a linear cone.
\end{proposition}

\begin{proposition}[{\cite[Proposition~1.5]{przyjalkowski/automorphisms}}]\label{proposition:WCI-QS-WF-implies-unique}
  Let \(X' \subset \mathbb{P}(\rho')\) and \(X'' \subset \mathbb{P}(\rho'')\) be quasi-smooth well-formed weighted complete intersections of multidegrees \(\mu'\) and \(\mu''\), respectively, such that \(X'\) and \(X''\) are not intersections with linear cones. Let us assume that \(X' \simeq X''\), and \(\dim(X') \geqslant 3\). Then \(\rho' = \rho''\), and \(\mu' = \mu''\).
\end{proposition}

\begin{remark}\label{remark:WCI-QS-WF-implies-unique}
  The assertion of Proposition~\ref{proposition:WCI-QS-WF-implies-unique} also holds for smooth well-formed Fano weighted complete intersections of dimension 2, and if we drop the dimension assumption, then \(X \simeq \mathbb{P}^1\) is contained in two families of complete intersections: namely, \(\mathbb{P}^1\) itself and conics (see~\cite[Lemma~2.6]{przyjalkowski/automorphisms}).
\end{remark}

\begin{lemma}[{\cite[Lemma~18.14]{ianofletcher/weighted}},~{\cite[Proposition~3.1(1)]{chen/quasismooth}}]\label{lemma:WCI-QS-WF-degrees-bound}
  Let \(X\) be a quasi-smooth well-formed weighted complete intersection in \(\mathbb{P}(a_0, \ldots, a_N)\) of multidegree \((d_1, \ldots, d_c)\) which is not an intersection with a linear cone. The following estimates hold: \(d_{c - j} > a_{N - j}\) for all \(j = 1, \ldots, c\), and \(d_c \geqslant 2 a_N\).
\end{lemma}

\begin{proposition}[{\cite[Remark~4.2]{okada/rationality}},~{\cite[Proposition~2.3]{pizzato/nonvanishing}}]
  Let \(X \subset \mathbb{P}(\rho)\) be a quasi-smooth well-formed weighted complete intersection of dimension at least 3. Then the divisor class group \(\Cl(X)\) is generated by the class of the divisorial sheaf \(\mathcal{O}_X(1)\).
\end{proposition}

\begin{proposition}[{\cite[Theorem~3.3.4]{dolgachev/weighted}},~{\cite[\nopp 6.14]{ianofletcher/weighted}}]\label{proposition:WCI-QS-WF-adjunction}
  Let \(X \subset \mathbb{P}(a_0, \ldots, a_N)\) be a quasi-smooth well-formed weighted complete intersection of dimension at least 2 and multidegree \((d_1, \ldots, d_c)\), and \(\omega_X\) be the dualising sheaf of \(X\). Then the following identity holds: \(\omega_X \simeq \mathcal{O}_X (\sum_{j = 1}^c d_j - \sum_{i = 0}^N a_i)\).
\end{proposition}

\begin{corollary}[{\cite[Corollary~2.8]{przyjalkowski/automorphisms}}]\label{corollary:WCI-fano-index}
  Let \(X \subset \mathbb{P}(a_0, \ldots, a_N)\) be a smooth well-formed Fano weighted complete intersection of dimension at least 2 and multidegree \((d_1, \ldots, d_c)\). Then index of \(X\) equals to \(i_X = \sum_{i = 0}^N a_i - \sum_{j = 1}^c d_j\).
\end{corollary}

\begin{corollary}[see the proof of~{\cite[Theorem~2.7]{przyjalkowski/hodge}}]
  Let \(X \subset \mathbb{P}(a_0, \ldots, a_N)\) be a smooth well-formed Fano weighted complete intersection of multidegree \((d_1, \ldots, d_c)\) and dimension at least 2, and \(i_X\) be its index. Then anticanonical degree of \(X\) is equal to \((\prod_{j = 1}^c d_j / \prod_{i = 0}^N a_i) \cdot i_X^{\dim(X)}\).
\end{corollary}

\begin{corollary}\label{corollary:WCI-fano-degree-one}
  Let \(X \subset \mathbb{P}(a_0, \ldots, a_N)\) be a smooth well-formed Fano weighted complete intersection of multidegree \((d_1, \ldots, d_c)\). Then we have \((-K_X)^{\dim(X)} = 1\) if and only if \(i_X = 1\), and \(\prod_{i = 0}^N a_i = \prod_{j = 1}^c d_j\).
\end{corollary}

\begin{proposition}[{\cite[Theorem~1.2]{pizzato/nonvanishing}},~{\cite[Corollary~3.4]{przyjalkowski/codimension}}]\label{proposition:WCI-fano-CY-O1-dim-bound}
  Let \(X \subset \mathbb{P}(\rho)\) be a smooth well-formed weighted complete intersection which is not an intersection with a linear cone. The following assertions hold.
  \begin{enumerate}
  \item If \(X\) is Calabi--Yau, then we have \(\dim(H^0(X, \mathcal{O}_X(1))) \geqslant \codim(X)\), and the equality holds if and only if \(X \subset \mathbb{P}(1^{(c)}, 2^{(c)}, 3^{(c)})\) is a weighted complete intersection of multidegree \((6^{(c)})\).
  \item If \(X\) is Fano, then we have \(\dim(H^0(X, \mathcal{O}_X(1))) \geqslant \codim(X) + i_X\), and the equality holds if and only if \(X \subset \mathbb{P}(1^{(i_X + c)}, 2^{(c)}, 3^{(c)})\) is a weighted complete intersection of multidegree \((6^{(c)})\).
  \end{enumerate}
\end{proposition}

\subsection{Combinatorics of weighted complete intersections}

Now we discuss how certain geometrical properties of weighted complete intersections affect the combinatorics of its weights and degrees.

\begin{definition}[{\cite[Example~1.5]{flajolet/combinatorics}}]\label{definition:denumerant}
  Let \(\rho = (a_0, \ldots, a_N)\) be a tuple of positive integers. For any \(n \in \mathbb{Z}_{\geqslant 0}\) the \emph{denumerant} \(D(\rho; n)\) is the number of solutions \((\alpha_0, \ldots, \alpha_N) \in \mathbb{Z}_{\geqslant 0}^{N + 1}\) of the equation \(\sum_{i = 0}^N a_i \alpha_i = n\).
\end{definition}

\begin{lemma}[{\cite[Proposition~I.1, Example~1.5]{flajolet/combinatorics}}]\label{lemma:denumerants-series}
  Let \(\rho = (a_0, \ldots, a_N)\) be a tuple of positive integers. The generating series of denumerants \(D(\rho; \bullet)(T)\) has the following form:
  \[
    D(\rho; \bullet)(T) = \sum_{n = 0}^{\infty} D(\rho; n) T^n =
    \frac{1}{\prod_{i = 0}^N (1 - T^{a_i})}.
  \]
\end{lemma}

For example, there exist necessary and sufficient conditions for quasi-smoothness of a weighted complete intersection in terms of its weights and degrees. Let us introduce the preliminary notations.

\begin{definition}\label{definition:Q1-Q2-type}
  Let \(\rho = (a_0, \ldots, a_N)\) and \(\mu = (d_1, \ldots, d_m)\) be tuples of positive integers. For any subset \(I \subset \{0, \ldots, N\}\) we introduce the following notation:
  \[
    \rho_I = (a_{i_1}, \ldots, a_{i_s}), \quad
    I = \{i_1, \ldots, i_s\}, \quad i_1 < \cdots < i_s.
  \]
  For any subsets \(I \subset \{0, \ldots, N\}\) and \(J \subset \{1, \ldots, m\}\) we also put \(I^{\circ} = \{0, \ldots, N\} \setminus I\) and \(J^{\circ} = \{1, \ldots, m\} \setminus J\).

  We say that a subset \(I \subset \{0, \ldots, N\}\) is
  \begin{itemize}
  \item \emph{of type} \((Q_1)\) if the following condition holds:
    \begin{itemize}
    \item there exists a subset \(J \subset \{1, \ldots, m\}\) such that \(\vert J \vert = \vert I \vert\), and \(D(\rho_I; d_j) > 0\) for any \(j \in J\).
    \end{itemize}
  \item \emph{of type} \((Q_2)\) if the following two conditions hold:
    \begin{itemize}
    \item there exists a subset \(J \subset \{1, \ldots, m\}\) such that \(\vert J \vert < \vert I \vert\), and \(D(\rho_I; d_j) > 0\) for any \(j \in J\);
    \item for any \(\mu = 1, \ldots, \vert I \vert - \vert J \vert\) there exists a subset
      \(
        E_{\mu} = \{e_{j, \mu} \mid j \in J^{\circ}\} \subset I^{\circ}
      \)
      of cardinality \(\vert J^{\circ} \vert\) such that we have \(D(\rho_I; d_j - a_{e_{\mu, j}}) > 0\) for any \(j \in J^{\circ}\), and for any subset \(G \subset J^{\circ}\) the following estimate holds:
      \(
      \vert \{e_{g, \mu} \mid g \in G, \;
      \mu = 1, \ldots, \vert I \vert - \vert J \vert\} \vert \geqslant
        \vert I \vert - \vert J \vert + \vert G \vert - 1.
      \)
    \end{itemize}
  \end{itemize}
\end{definition}

\begin{definition}
  Let \(\rho = (a_0, \ldots, a_N)\) and \(\mu\) be tuples of positive integers. We refer to \((\rho; \mu)\) as a \emph{consistent pair} if any subset \(I \subset \{0, \ldots, N\}\) is of type \((Q_1)\) or \((Q_2)\) with respect to Definition~\ref{definition:Q1-Q2-type}.  
\end{definition}

\begin{proposition}[{\cite[Proposition~3.1]{pizzato/nonvanishing}}]\label{proposition:WCI-QS-criterion}
  Assume that there exists a weighted complete intersection \(X \subset \mathbb{P}(\rho)\) of multidegree \(\mu\) which is not an intersection with a linear cone. The following assertions hold.
  \begin{enumerate}
  \item If \(X\) is quasi-smooth, then the pair \((\rho; \mu)\) is consistent.
  \item Assume that the pair \((\rho; \mu)\) is consistent, and \(X\) is general in the family \(\mathcal{F}^{\rho}_{\mu}\) of weighted complete intersections. Then \(X\) is quasi-smooth. 
  \end{enumerate}
\end{proposition}

Another interesting geometrical property is the non-intersection with a given coordinate stratum in a weighted projective space (for example, cf. Lemma~\ref{lemma:WPS-WF-singular-locus}), which allows the following combinatorial description.

\begin{definition}
  Let \(\rho = (a_0, \ldots, a_N)\) and \(\mu = (d_1, \ldots, d_m)\) be tuples of positive integers. A subset \(I \subset \{0, \ldots, N\}\) \emph{is of type \((Q_1^+)\)} if \(\vert I \vert \leqslant m\), and any subset \(\widetilde{I} \subset I\) is of type \((Q_1)\).
\end{definition}

\begin{lemma}\label{lemma:WCI-stratum-intersection}
  Let \(X \subset \mathbb{P}(\rho) = \Proj(R^{\rho})\) be a weighted complete intersection of multidegree \(\mu\). Put \(\rho = (a_0, \ldots, a_N)\), and let \(I \subset \{0, \ldots, N\}\) be a subset. We also use the following notation:
  \[
    \Lambda_I = \{X_i = 0 \mid i \not \in I\} \subset \mathbb{P}(\rho), \quad
    R^{\rho} = \Bbbk[X_0, \ldots, X_N].
  \]
  If \(X \cap \Lambda_I = \varnothing\), then the subset \(I\) is of type \((Q_1^+)\).
\end{lemma}

\begin{proof}
  Put \(\mu = (d_1, \ldots, d_c)\). The defining ideal \(I_X \subset R^{\rho}\) (see Definition~\ref{definition:defining-ideal}) is generated by a regular sequence \((f_1, \ldots, f_c)\) of weighted homogeneous polynomials of degrees \(d_j = \deg(f_j)\) for all \(j = 1, \ldots, c\). For any subset \(\widetilde{I} \subset I\) we introduce the monomial ideal
  \(
    S_{\widetilde{I}} = \langle \{X_i \mid i \not \in \widetilde{I}\} \rangle \subset R^{\rho}
  \).

  By assumption for any subset \(\widetilde{I} \subset I\) we have \(\height(I_X + S_{\widetilde{I}}) = \vert \widetilde{I} \vert\). Consequently, for any subset \(\widetilde{I} \subset I\) there should exist a subset \(J_{\widetilde{I}} \subset \{1, \ldots, c\}\) of cardinality \(\vert \widetilde{I} \vert\) such that the polynomial \(f_j\) contains a monomial in variables \(\{X_i \mid i \in \widetilde{I}\}\) for any \(j \in J_{\widetilde{I}}\), so we are done.
\end{proof}

\begin{definition}\label{definition:strictly-regular-pair}
  Let \(\rho = (a_0, \ldots, a_N)\) and \(\mu\) be tuples of positive integers. We refer to \((\rho; \mu)\) as a \emph{strictly regular pair} if any subset \(I \subset \{0, \ldots, N\}\) such that \(\gcd(\{a_i \mid i \in I\}) > 1\) is of type \((Q_1^+)\).  
\end{definition}

\begin{corollary}[cf.~{\cite[Lemma~2.15]{przyjalkowski/bounds}},~{\cite[Proposition~4.1]{chen/quasismooth}}]\label{corollary:WCI-smooth-WF-combinatorics}
  Let \(X \subset \mathbb{P}(\rho)\) be a smooth well-formed weighted complete intersection of multidegree \(\mu\) which is not an intersection with a linear cone. Then the pair \((\rho; \mu)\) is consistent and strictly regular.
\end{corollary}

\begin{proof}
  By Corollary~\ref{corollary:WCI-smooth-WF-criterion} \(X\) is quasi-smooth and does not intersect \(\Sing(\mathbb{P}(\rho))\). Then the pair \((\rho; \mu)\) is consistent by Proposition~\ref{proposition:WCI-QS-criterion}. Lemma~\ref{lemma:WPS-WF-singular-locus} explicitly describes \(\Sing(\mathbb{P}(\rho))\) as a union of coordinate strata. Then the pair \((\rho; \mu)\) is strictly regular by Lemma~\ref{lemma:WCI-stratum-intersection}.
\end{proof}

\begin{corollary}\label{corollary:WCI-amenable-combinatorics}
  Let \(X \subset \mathbb{P}(a_0, \ldots, a_N)\) be an amenable weighted complete intersection of multidegree \(\mu\). Then the subset \(\{i \mid a_i > 1\}\) is of type \((Q_1^+)\).
\end{corollary}

\begin{proof}
  By Definition~\ref{definition:WCI-amenable} we should have \(\Bs(\vert \mathcal{O}_X(1) \vert) = \varnothing\). Corollary~\ref{corollary:WCI-QS-O1-base-locus} implies that \(X \cap \Bs(\vert \mathcal{O}_{\mathbb{P}(\rho)}(1) \vert) = \varnothing\). Note that \(\Bs(\vert \mathcal{O}_{\mathbb{P}(\rho)}(1)) \vert\) can be described as the coordinate stratum in \(\mathbb{P}(\rho)\) (see Lemma~\ref{lemma:WPS-WF-On-sheaves}). Then the statement follows from Lemma~\ref{lemma:WCI-stratum-intersection}.
\end{proof}

It is useful to compare Definition~\ref{definition:strictly-regular-pair} with its weaker version which was introduced in~\cite{pizzato/nonvanishing}.

\begin{definition}[{\cite[Definition~4.1]{pizzato/nonvanishing}}]
  Let \(\rho = (a_0, \ldots, a_N)\) and \(\mu = (d_1, \ldots, d_m)\) be tuples of positive integers. We refer to \((\rho; \mu)\) as a \emph{regular pair} if for any subset \(I \subset \{0, \ldots, N\}\) such that \(\gcd(\{a_i \mid i \in I\}) > 1\) there exists a subset \(J \subset \{1, \ldots, c\}\) of cardinality \(\vert I \vert\) such that \(\gcd(\{a_i \mid i \in I\})\) divides \(\gcd(\{d_j \mid j \in J\})\).
\end{definition}

We can restate Corollaries~\ref{corollary:WCI-smooth-WF-combinatorics} and~\ref{corollary:WCI-amenable-combinatorics} in these terms, which is more suitable for explicit computations.

\begin{notation}
  For any \(\rho = (a_0, \ldots, a_N) \in \mathbb{Z}_{> 1}^{N + 1}\) and \(h \in \mathbb{Z}_{> 0}\) we put \(I(h) = \{i \mid a_i \in h \mathbb{Z}, \; a_i > 1\}\).
\end{notation}

\begin{remark}
  We have \(I(h) = \{i \mid a_i \in h \mathbb{Z}\}\) for \(h > 1\), and \(I(1) = \{i \mid a_i > 1\}\) for \(h = 1\).
\end{remark}

\begin{lemma}  
  Let \((\rho; \mu)\) be a consistent and regular pair. Then \(I(h)\) is of type \((Q_1)\) for any \(h \in \mathbb{Z}_{> 1}\).
\end{lemma}

\begin{proof}
  Put \(\mu = (d_1, \ldots, d_m)\) and \(J(h) = \{j \mid d_j \in h \mathbb{Z}\}\). We claim that the subset \(I(h)\) is of type \((Q_1)\). On the contrary, suppose that \((Q_2)\) holds. Firstly, there exists a subset \(J \subset \{1, \ldots, m\}\) of cardinality less than \(\vert I(h) \vert\) such that for all \(j \in J\) we have \(D(\rho_{I(h)}; d_j) > 0\). Secondly, for every \(j \in J^{\circ} = \{1, \ldots, m\} \setminus J\) and \(\mu = 1, \ldots, \vert I(h) \vert - \vert J \vert\) there are integers \(e_{j, \mu} \in \{0, \ldots, N\} \setminus I\) such that \(D(\rho_{I(h)}; d_j - a_{e_{j, \mu}} ) > 0\). By definition \(a_{e_{j, \mu}}\) are not divisible by \(h\), then \(d_j\) is not divisible by \(h\) for any \(j \in J^{\circ}\). In other words, we have \(m - \vert J \vert \leqslant m - \vert J(h) \vert\), i.e., \(\vert J(h) \vert \leqslant \vert J \vert\). From the regularity assumption we have \(\vert I(h) \vert \leqslant \vert J(h) \vert\), hence \(\vert I(h) \vert \leqslant \vert J \vert\), which is a contradiction.
\end{proof}

\begin{lemma}\label{lemma:Q1-type-implies-Q1-plus-type}
  Let \(\rho = (a_0, \ldots, a_N)\) and \(\mu = (d_1, \ldots, d_m)\) be tuples of positive integers such that the pair \((\rho; \mu)\) is consistent and regular.

  Assume that \(\vert I(h) \vert \leqslant m\), and \(I(h)\) is of type \((Q_1)\). Then \(I(h)\) is of type \((Q_1^+)\).
\end{lemma}

\begin{proof}
  Put \(I = I(h)\). Let \(\widetilde{I} \subset I\) be any non-trivial subset. We claim that \(\widetilde{I}\) is of type \((Q_1)\). On the contrary, suppose that it is of type \((Q_2)\). Firstly, there exists a subset \(\widetilde{J} \subset \{1, \ldots, m\}\), of cardinality less than \(\vert \widetilde{I} \vert\) such that for all \(j \in \widetilde{J}\) we have \(D(\rho_{\widetilde{I}}; d_j) > 0\). Secondly, for every \(j \in \widetilde{J}^{\circ} = \{1, \ldots, m\} \setminus \widetilde{J}\) and \(\mu = 1, \ldots, \vert \widetilde{I} \vert - \vert \widetilde{J} \vert\) there are integers \(e_{j, \mu} \in \{0, \ldots, N\} \setminus \widetilde{I}\) such that \(D(\rho_{\widetilde{I}}; d_j - a_{e_{j, \mu}}) > 0\), and for any \(G \subset \widetilde{J}\) we have
  \[
    \vert \{e_{g, \mu} \mid g \in G, \;
    \mu = 1, \ldots, \vert \widetilde{I} \vert - \vert \widetilde{J} \vert\} \vert \geqslant
    \vert \widetilde{I} \vert - \vert \widetilde{J} \vert + \vert G \vert - 1.
  \]

  By assumption there exists a subset \(\widetilde{G} \subset \widetilde{J}^{\circ}\) of cardinality \(\vert I \vert - \vert \widetilde{J} \vert\) such that \(D(\rho_I; d_j) > 0\) for any \(j \in \widetilde{G}\), since \(I\) is of type \((Q_1)\), and \(\vert I \vert \leqslant m\). But then by the definition of the subset \(I = I(h)\) we have \(e_{j, \mu} \in I\) for any \(j \in \widetilde{J}\). We obtain a contradiction:
  \(
  \vert \widetilde{I} \vert + \vert \{e_{g, \mu} \mid g \in \widetilde{G}, \;
  \mu = 1, \ldots, \vert \widetilde{I} \vert - \vert \widetilde{J} \vert\} \vert \geqslant
  2 (\vert \widetilde{I} \vert - \vert \widetilde{J} \vert) - 1 + \vert I \vert > \vert I \vert.
  \) 
\end{proof}

\begin{corollary}
  A consistent and regular pair is strictly regular.
\end{corollary}

\begin{corollary}[cf.~Corollary~\ref{corollary:WCI-amenable-combinatorics} and Example~\ref{example:WCI-non-amenable}]
  Let \(X \subset \mathbb{P}(\rho)\) be a quasi-smooth well-formed weighted complete intersection of multidegree \(\mu\) which is non-degenerate. The following assumptions are equivalent:
  \begin{itemize}
  \item the subset \(I(1)\) is of type \((Q_1)\), the pair \((\rho; \mu)\) is regular, and \(\dim(\vert \mathcal{O}_X(1) \vert) \geqslant \dim(X)\);
  \item the subset \(I(1)\) is of type \((Q_1^+)\).
  \end{itemize}
\end{corollary}

\begin{proof}
  Assume that the subset \(I(1)\) is of type \((Q_1)\), the pair \((\rho; \mu)\) is regular, and \(\dim(\vert \mathcal{O}_X(1) \vert) \geqslant \dim(X)\). Lemma~\ref{lemma:WCI-QS-WF-O1-dimension} implies that \(\vert I(1) \vert \leqslant c\). Then the statement follows from Lemma~\ref{lemma:Q1-type-implies-Q1-plus-type}.

  Conversely, let us assume that the subset \(I(1)\) is of type \((Q_1^+)\). Then by definition the pair \((\rho; \mu)\) is regular, the subset \(I(1)\) is of type \((Q_1)\), and \(\vert I(1) \vert \leqslant c\). Then Lemma~\ref{lemma:WCI-QS-WF-O1-dimension} implies that \(\dim(\vert \mathcal{O}_X(1) \vert) \geqslant \dim(X)\).
\end{proof}

\section{Series and semi-series of smooth well-formed Fano weighted complete intersections}\label{section:series-semiseries}

\subsection{Structure of series and semi-series}\label{subsection:structure}

Firstly, we prove that the partition into series and semi-series (see Definition~\ref{definition:series-semiseries}) is correctly defined.

\begin{lemma}\label{lemma:WCI-fano-hyperplane-section-inverse}
  Let \(X \subset \mathbb{P}(\rho)\) be a smooth well-formed Fano weighted complete intersection of multidegree \(\mu\) which is not an intersection with a linear cone. Then a general element \(X'\) of the family \(\mathcal{F}^{(1; \rho)}_{\mu}\) of weighted complete intersections is also a smooth well-formed Fano weighted complete intersection which is not an intersection with a linear cone. If \(\dim(X) > 1\), then index of \(X'\) equals \(i_{X'} = i_X + 1\).
\end{lemma}

\begin{proof}
  By definition \(\mathbb{P}(1; \rho)\) is a well-formed weighted projective space. Note that elements of the family \(\mathcal{F}^{(1; \rho)}_{\mu} \colon T \rightarrow B\) are not intersections with linear cones. Now we only have to prove that the subset
  \[
    B' = \{X' \in B \mid X' \text{ is smooth and well-formed}\}
  \]
  is non-empty and open (the Fano condition would follow from Proposition~\ref{proposition:WCI-QS-WF-adjunction}).

  By Corollary~\ref{corollary:WCI-smooth-WF-criterion} we can present \(B'\) in the form \(B' = B'_1 \cap B'_2\), where
  \[
    B'_1 = \{X' \in B \mid X' \text{ is quasi-smooth}\}, \quad
    B'_2 = \{X' \in B \mid X' \cap \Sing(\mathbb{P}(1; \rho)) = \varnothing\}.
  \]
  These subsets are open in \(B\) (see the proof of Corollary~\ref{corollary:WCI-smoothness-genericity}).

  Note that \(B'_2\) is not empty: it contains a projective cone over \(X\) by Lemma~\ref{lemma:WPS-WF-singular-locus}. Proposition~\ref{proposition:WCI-QS-criterion} implies that \(B'_1\) is non-empty as well. Since \(B\) is irreducible by Lemma~\ref{lemma:WCI-family-irreducible-base}, we have \(B'_1 \cap B'_2 \neq \varnothing\).
\end{proof}

\begin{corollary}\label{corollary:WCI-fano-upper-index}
  Let \(\mathcal{F} \in \Theta\) be a family of weighted complete intersections. Then \(\mathcal{F}^{(l)}_{(0)} \in \Theta\) for any \(l \in \mathbb{Z}_{\geqslant 0}\).
\end{corollary}

\begin{lemma}[{see the proof of Theorem~\ref{theorem:WCI-fano-small-variance}}]\label{lemma:WCI-fano-hyperplane-section}
  Let \(X \subset \mathbb{P}(1; \rho)\) be a smooth well-formed Fano weighted complete intersection of multidegree \(\mu\) and dimension at least 2 which is not an intersection with a linear cone. Then a general element \(X'\) of the family of weighted complete intersections \(\mathcal{F}^{\rho}_{\mu}\) is also a smooth well-formed weighted complete intersection which is not an intersection with a linear cone.

  If \(i_X > 1\), then \(X'\) is Fano. Moreover, if \(\dim(X) > 2\), then index of \(X'\) equals \(i_{X'} = i_X - 1\).
\end{lemma}

\begin{proof}
  By~\cite[Theorem~1.2]{pizzato/nonvanishing} the linear system \(\vert \mathcal{O}_X(1) \vert\) is not empty. Let \(X'\) be a general divisor from \(\vert \mathcal{O}_X(1) \vert\). Then \(X'\) is smooth by~\cite[Theorem~1.2]{pizzato/nonvanishing}. Moreover, \(X'\) is a well-formed weighted complete intersection of multidegree \(\mu\) in \(\mathbb{P}(\rho)\) by~\cite[Lemma~3.3]{przyjalkowski/codimension}, and \(X'\) is not an intersection with a linear cone. Now we only have to apply Proposition~\ref{proposition:WCI-QS-WF-adjunction}.
\end{proof}

\begin{lemma}\label{lemma:WCI-amenability-necessary}
  Let \(X \subset \mathbb{P}(\rho)\) be a quasi-smooth subscheme in a well-formed weighted projective space. Put \(\rho = (a_0, \ldots, a_N)\), where \(a_i \neq 2\) for any \(i = 0, \ldots, N\). Assume that there exists an element \(Q \in \vert \mathcal{O}_{\mathbb{P}(\rho)}(2) \vert\) such that \(X \cap Q\) is also quasi-smooth, and \(\codim_X(X \cap Q) = 1\). Then \(X\) does not intersect \(\Bs(\vert \mathcal{O}_{\mathbb{P}(\rho)}(1) \vert)\).
\end{lemma}

\begin{proof}
  Lemma~\ref{lemma:WPS-WF-On-sheaves} implies that the base locus of the linear system \(\vert \mathcal{O}_{\mathbb{P}(\rho)}(1) \vert\) can be described as
  \[
    \Bs(\vert \mathcal{O}_{\mathbb{P}(\rho)}(1) \vert) = \{X_i = 0 \mid a_i = 1\} \subset
    \mathbb{P}(\rho) = \Proj(R^{\rho}), \quad R^{\rho} = \Bbbk[X_0, \ldots, X_N].
  \]
  Let \(I_X \subset R^{\rho}\) be the defining ideal of \(X\) (see Definition~\ref{definition:defining-ideal}), and \(f \in R^{\rho}\) be a quadratic polynomial in the variables \(\{X_i \mid a_i = 1\}\) that represents \(Q\) (see Lemma~\ref{lemma:WPS-WF-On-sheaves}). It is clear that \(f\) vanishes on the stratum
  \(
  \{X_i = 0 \mid a_i = 1\} \subset \mathbb{A}^{N + 1}
  \)
  and so do its partial derivatives. But by assumption the subscheme
  \(
  \Spec(R^{\rho} / (I_X + (f))) \subset \mathbb{A}^{N + 1}
  \)
  is smooth outside the origin, hence \(X\) does not intersect \(\Bs(\vert \mathcal{O}_{\mathbb{P}(\rho)}(1) \vert)\).
\end{proof}

\begin{corollary}\label{corollary:WCI-amenability-necessary}
  Let \(X \subset \mathbb{P}(\rho)\) be a quasi-smooth well-formed weighted complete intersection which is non-degenerate. Put \(\rho = (a_0, \ldots, a_N)\), where \(a_i \neq 2\) for any \(i = 0, \ldots, N\). Assume that there exists a quasi-smooth element \(Q \in \vert \mathcal{O}_X(2) \vert\). Then \(X\) and \(Q\) are amenable weighted complete intersections.
\end{corollary}

\begin{proof}
  Lemma~\ref{lemma:WCI-QS-WF-projective-normality} implies that the element \(Q \in \vert \mathcal{O}_X(2) \vert\) can be lifted to an element \(\widetilde{Q} \in \vert \mathcal{O}_{\mathbb{P}(\rho)}(2) \vert\) such that \(\widetilde{Q} \cap X = Q\). Then Lemma~\ref{lemma:WCI-amenability-necessary} implies that \(X\) does not intersect \(\Bs(\vert \mathcal{O}_{\mathbb{P}(\rho)}(1) \vert)\) (and so does not \(Q \subset X\)). But then Corollary~\ref{corollary:WCI-QS-O1-base-locus} implies that \(X\) and \(Q\) are amenable.
\end{proof}

\begin{lemma}
  Let \(X \subset \mathbb{P}(\rho)\) be a weighted complete intersection in a well-formed weighted projective space of multidegree \(\mu\) which is not an intersection with a linear cone. Put \(\rho = (a_0, \ldots, a_N)\), where \(a_i \neq 2\) for any \(i = 0, \ldots, N\). Assume that \(X\) does not intersect \(\Bs(\vert \mathcal{O}_{\mathbb{P}(\rho)}(1) \vert)\).

  Then general elements \(X' \in \mathcal{F}^{\rho}_{\mu}\) and \(X'' \in \mathcal{F}^{(1, 1; \rho)}_{(2; \mu)}\) are strictly amenable.
\end{lemma}

\begin{proof}
  By assumption and Lemma~\ref{lemma:WPS-WF-On-sheaves} \(X\) does not intersect the coordinate stratum 
  \[
    \Bs(\vert \mathcal{O}_{\mathbb{P}(\rho)}(1) \vert) =
    \{X_i = 0 \mid a_i = 1\} \subset \mathbb{P}(\rho) = \Proj(R^{\rho}), \quad
    R^{\rho} = \Bbbk[X_0, \ldots, X_N].
  \]
  Lemma~\ref{lemma:WCI-stratum-intersection} implies that the subset \(\{i \mid a_i > 1\}\) is of type \((Q_1^+)\). We can apply Proposition~\ref{proposition:WCI-QS-criterion} to conclude that a general element \(X' \in \mathcal{F}^{\rho}_{\mu}\) is quasi-smooth. Moreover, \(X'\) also does not intersect \(\Bs(\vert \mathcal{O}_{\mathbb{P}(\rho)}(1) \vert)\) by Lemma~\ref{lemma:hilbert-series-semicontinuity}. Then it is smooth and well-formed by Corollary~\ref{corollary:WCI-smooth-WF-criterion}, hence  \(X'\) is amenable by Corollary~\ref{corollary:WCI-QS-O1-base-locus}.

  Consider a general element \(\widetilde{X} \in \mathcal{F}^{(1, 1; \rho)}_{\mu}\), which is smooth and well-formed by Lemma~\ref{lemma:WCI-fano-hyperplane-section-inverse}, hence is amenable by Corollary~\ref{corollary:WCI-QS-O1-base-locus}. It is clear that \(\vert \mathcal{O}_{\widetilde{X}}(2) \vert \neq \varnothing\), so let \(Q \in \vert \mathcal{O}_{\widetilde{X}}(2) \vert\) be a general element. Lemma~\ref{lemma:WCI-QS-WF-projective-normality} implies that \(Q\) can be lifted up to a general element \(X'' \in \mathcal{F}^{(1, 1; \rho)}_{(2; \mu)}\). But we have \(X'' \subset \widetilde{X}\), hence \(X'' \cap \Bs(\vert \mathcal{O}_{\mathbb{P}(1, 1; \rho)}(1) \vert) \subset \mathbb{P}(1, 1; \rho)\) is empty as well. Then we can repeat the same argument for \(X''\).
\end{proof}

\begin{corollary}
  Let \(X \subset \mathbb{P}(\rho)\) be a strictly amenable weighted complete intersection of multidegree \(\mu\). Then a general element \(X' \in \mathcal{F}^{(1, 1; \rho)}_{(2; \mu)}\) is strictly amenable.
\end{corollary}

\begin{corollary}\label{corollary:WCI-fano-lower-index}
  Let \(\mathcal{F} \in \Theta\) be a family of weighted complete intersections. Then for any \(m \in \mathbb{Z}_{> 0}\) we have \(\mathcal{F}^{(0)}_{(m)} \in \Theta\) if and only if \(\mathcal{F}\) contains a strictly amenable weighted complete intersection.
\end{corollary}

\begin{lemma}\label{lemma:WCI-amenability-elimination}
  Let \(X \subset \mathbb{P}(1, 1; \rho)\) be a smooth well-formed Fano weighted complete intersection of multidegree \((2; \mu)\) which is not an intersection with a linear cone. Then a general element \(X'\) of the family of weighted complete intersections \(\mathcal{F}^{\rho}_{\mu}\) is a strictly amenable Fano weighted complete intersection. If \(\dim(X) > 1\), then index of \(X'\) coincide with index of \(X\).
\end{lemma}

\begin{proof}
  By definition there exist elements \(Q \in \vert \mathcal{O}_{\mathbb{P}(\rho)}(2) \vert\) and \(X' \in \mathcal{F}^{(1, 1; \rho)}_{\mu}\) such that \(Q \cap X' = X\). Then \(X'\) is a strictly amenable weighted complete intersection by Corollary~\ref{corollary:WCI-amenability-necessary}, and is Fano by Corollary~\ref{corollary:WCI-fano-index}. Now we only have to apply Lemma~\ref{lemma:WCI-fano-hyperplane-section} twice.
\end{proof}

\begin{lemma}\label{lemma:WCI-fano-generator}
  Let \(\mathcal{F}^{\rho}_{\mu} \in \Theta\) be a family of weighted complete intersections. Put
  \[
    \rho = (a_0, \ldots, a_N), \quad \mu = (d_1, \ldots, d_c); \quad
    I(\mathcal{F}^{\rho}_{\mu}) = \sum_{i = 0}^N a_i - \sum_{j = 1}^c d_j, \quad
    s_2(\mathcal{F}^{\rho}_{\mu}) = \vert \{j \mid d_j > 2\} \vert.
  \]
  Then the family \(\mathcal{F}^{\rho}_{\mu}\) generates the series or semi-series if and only if we have \(I(\mathcal{F}^{\rho}_{\mu}) = 1\) and \(s_2(\mathcal{F}^{\rho}_{\mu}) = c\).
\end{lemma}

\begin{proof}
  Let us prove the ``if'' part. On the one hand, the assumption \(s_2(\mathcal{F}^{\rho}_{\mu}) = c\) implies that we cannot have \(\mathcal{F}^{\rho}_{\mu} = \mathcal{G}^{(0)}_{(1)}\) for any family \(\mathcal{G} \in \Theta\). On the other hand, by Corollary~\ref{corollary:WCI-fano-index} a general member of the family \(\mathcal{F}^{\rho}_{\mu}\) has Fano index 1. Lemma~\ref{lemma:WCI-fano-hyperplane-section} implies that there exists a family \(\mathcal{G}\) of weighted complete intersections such that \(\mathcal{F}^{\rho}_{\mu} = \mathcal{G}^{(1)}_{(0)}\), but its general element is a Calabi--Yau variety, hence \(\mathcal{G} \not \in \Theta\).

  Let us prove the ``only if'' part. Let \(X \in \mathcal{F}^{\rho}_{\mu}\) be a smooth well-formed Fano weighted complete intersection. If \(\dim(X) > 1\), then we can apply Lemma~\ref{lemma:WCI-fano-hyperplane-section} a sufficient number of times until either \(i_X = 1\), or \(\dim(X) = 1\). In the latter case \(X\) is either a projective line or a smooth conic by~\cite[Lemma~2.6]{przyjalkowski/automorphisms}, and the corresponding families can be obtained from the family corresponding to a point.

  Consequently, we can assume that \(I(\mathcal{F}^{\rho}_{\mu}) = 1\) by Corollary~\ref{corollary:WCI-fano-index}. If \(X\) is not amenable, then we have \(s_2(X) = \codim(X) = c\) by Corollary~\ref{corollary:WCI-amenability-necessary}. If \(X\) is amenable, we can apply Lemma~\ref{lemma:WCI-amenability-elimination} a sufficient number of times until we obtain \(s_2(\mathcal{F}) = \codim(X) = c\).
\end{proof}

\begin{corollary}\label{corollary:WCI-fano-structure}
  Any two different series or semi-series of smooth well-formed Fano weighted complete intersections do not intersect.
\end{corollary}

\begin{proof}
  Let \(\mathfrak{S}(\mathcal{F}')\) and \(\mathfrak{S}(\mathcal{F}'')\) be two series or semi-series generated by the families \(\mathcal{F}'\) and \(\mathcal{F}''\), respectively. Assume that there exists a family \(\mathcal{G} \in \Theta\) such that \(\mathcal{G} \in \mathfrak{S}(\mathcal{F}') \cap \mathfrak{S}(\mathcal{F}'')\). Then there exist numbers \(l',m,l'',m'' \in \mathbb{Z}_{\geqslant 0}\) such that \(\mathcal{G} = \mathcal{F}'^{(l')}_{(m')} = \mathcal{F}''^{(l'')}_{(m'')}\). Put \(\mathcal{F}' = \mathcal{F}^{\rho'}_{\mu'}\) and \(\mathcal{F}'' = \mathcal{F}^{\rho''}_{\mu''}\). Then we have the identities
  \[
    (1^{(l' + 2 m')}; \rho') = (1^{(l'' + 2 m'')}; \rho''), \quad
    (2^{(m')}; \mu') = (2^{(m'')}; \mu'').
  \]
  For any \(\mathcal{F} \in \Theta\) put \(I(\mathcal{F}) = \sum a_i - \sum d_j\) and \(s_2(\mathcal{F}) = \vert \{j \mid d_j > 2\} \vert\). We obtain the following identities:
  \[
    l' + I(\mathcal{F}') = l'' + I(\mathcal{F}''), \quad
    s_2(\mathcal{F}') = s_2(\mathcal{F}''), \quad m' + c' = m'' + c''.
  \]
  Lemma~\ref{lemma:WCI-fano-generator} implies that \(I(\mathcal{F}') = I(\mathcal{F}'') = 1\), \(s_2(\mathcal{F}') = c'\), and \(s_2(\mathcal{F}'') = c''\), hence \(l' = l''\), and \(m' = m''\). Consequently, we obtain that \(\mathcal{F}' = \mathcal{F}''\).
\end{proof}

\subsection{Bounds on the quadratic irregularity}\label{subsection:s2-bounds}

Our next step is to obtain the bound on the quadratic irregularity of a smooth well-formed Fano weighted complete intersection.

\begin{lemma}\label{lemma:CI-fano-variance-partition}
  Let \(X \subset \mathbb{P}^N\) be a smooth Fano complete intersection of multidegree \((d_1, \ldots, d_c)\) which is not an intersection with a linear cone. Then \((d_1 - 2, \ldots, d_c - 2)\) defines a weak partition of \(\var(X)\).
\end{lemma}

\begin{proof}
  We can assume that \(\dim(X) > 1\). Corollary~\ref{corollary:WCI-fano-index} implies that \(\var(X) = \sum_{j = 1}^c (d_j - 2)\). But \(X\) is not an intersection with a linear cone, hence \(d_j > 1\) for any \(j = 1, \ldots, c\).
\end{proof}

\begin{corollary}\label{corollary:CI-fano-s2-bound}
  Let \(X \subset \mathbb{P}^N\) be a smooth Fano complete intersection which is not an intersection with a linear cone. The following estimate holds: \(s_2(X) \leqslant \var(X)\).
\end{corollary}

\begin{lemma}\label{lemma:WCI-fano-variance-partition}
  Let \(X \subset \mathbb{P}(a_0, \ldots, a_N)\) be a smooth well-formed Fano weighted complete intersection of multidegree \((d_1, \ldots, d_c)\) which is not an intersection with a linear cone. Assume that \(\max_i(a_i) > 1\). Put
  \[
    \alpha_j = d_j - a_{\dim(X) + j} - 1, \quad j = 1, \ldots, c - 1; \quad
    \alpha_c = d_c - a_N - a_{\dim(X)} - 1.
  \]
  Then the following identity holds:
  \[
    \sum_{j = 1}^c \alpha_j = \sum_{i = c + i_X}^{\dim(X) - 1} a_i; \quad
    \alpha_j \geqslant 0, \quad j = 1, \ldots, c.
  \]
\end{lemma}

\begin{proof}
  Lemma~\ref{lemma:WCI-QS-WF-degrees-bound} implies that \(\alpha_j \geqslant 0\) for all \(j = 1, \ldots, c - 1\). Note that \(a_{\dim(X)} < a_N\): on the contrary, if \(a_{\dim(X)} = a_N\), then by Corollary~\ref{corollary:WCI-smooth-WF-combinatorics} at least \(c + 1\) degrees among \(d_1, \ldots, d_c\) is divisible by \(a_{\dim(X)}\), which is absurd. Then we obtain the inequality \(d_c \geqslant 2 a_N \geqslant a_N + a_{\dim(X)} + 1\) from Lemma~\ref{lemma:WCI-QS-WF-degrees-bound}. Moreover, Proposition~\ref{proposition:WCI-fano-CY-O1-dim-bound} implies that \(a_{c + i_X - 1} = 1\). Then the identity follows from Corollary~\ref{corollary:WCI-fano-index}.
\end{proof}

\begin{remark}\label{remark:WCI-fano-variance-partition}
  Let \(X \subset \mathbb{P}(\rho)\) be a smooth well-formed Fano weighted complete intersection which is not an intersection with a linear cone. Assume that \(\dim(\vert \mathcal{O}_X(1) \vert) \geqslant \dim(X)\), hence \(a_{\dim(X)} = 1\) by Lemma~\ref{lemma:WCI-QS-WF-O1-dimension}. Then Lemma~\ref{lemma:WCI-fano-variance-partition} states that the tuple \((\alpha_1, \ldots, \alpha_c)\) defines a weak composition of \(\var(X) - 1\). 
\end{remark}

\begin{lemma}\label{lemma:WCI-fano-O1-bound-variance}
  Let \(X \subset \mathbb{P}(a_0, \ldots, a_N)\) be a smooth well-formed Fano weighted complete intersection which is not an intersection with a linear cone. The following estimate holds: \(\dim(\vert \mathcal{O}_X(1) \vert) \geqslant N - 2 \var(X)\).
\end{lemma}

\begin{proof}
  Theorem~\ref{theorem:WCI-fano-small-variance} implies that we can always assume \(\var(X) > 0\). Consider the function \(W \colon \Theta \rightarrow \mathbb{Z}_{\geqslant 0}\) that maps a family \(\mathcal{G} \in \Theta\) of weighted complete intersections in \(\mathbb{P}(a'_0, \ldots, a'_{N'})\) of multidegree \((d_1, \ldots, d_c)\) to \(\vert \{i \mid a'_i > 1\} \vert\). Lemma~\ref{lemma:WCI-QS-WF-O1-dimension} implies that we can write \(W(\mathcal{G}) = N' - \dim(\vert \mathcal{O}_Y(1) \vert)\) for any smooth well-formed Fano weighted complete intersection \(Y \in \mathcal{G}\). Let us consider the partition \(\Theta = \bigsqcup \Theta^{(i, r)}\), where for any family \(\mathcal{G} \in \Theta^{(i, r)}\) and any smooth well-formed Fano weighted complete intersection \(Y \in \mathcal{G}\) we have \(\var(Y) = r\) and \(\sum a_i - \sum d_j = i\). By Proposition~\ref{proposition:WCI-fano-CY-O1-dim-bound} and Corollary~\ref{corollary:WCI-fano-index} the restriction of \(W\) to each subset \(\Theta^{(i, r)}\) has a unique maximum on the family \(\mathcal{F}_{(i, r)} \in \Theta^{(i, r)}\) of weighted complete intersection of multidegree \((6^{(r)})\) in \(\mathbb{P}(1^{(i + r)}, 2^{(r)}, 3^{(r)})\). It is clear that \(W(\mathcal{F}_{(i, r)}) = 2 r\). Then we obtain the bound \(W(\mathcal{F}) \leqslant 2 \var(X)\) for any family \(\mathcal{F} \in \Theta\) and any smooth well-formed Fano weighted complete intersection \(X \in \mathcal{F}\), so we are done.
\end{proof}

Now we derive Theorem~\ref{theorem:WCI-fano-s2-bound} from several partial cases.

\begin{lemma}\label{lemma:WCI-fano-s2-bound-exception}
  Let \(X \subset \mathbb{P}(a_0, \ldots, a_N)\) be a smooth well-formed Fano weighted complete intersection of multidegree \(\mu\) which is not an intersection with a linear cone. Assume that \(\max_i(a_i) = 2\). Then the following estimate holds: \(\var(X) \geqslant \codim(X)\).
\end{lemma}

\begin{proof}
  Note that the assumption \(\max_i(a_i) > 1\) implies \(\var(X) > 0\) by Theorem~\ref{theorem:WCI-fano-small-variance}. Put \(\mu = (d_1, \ldots, d_c)\). We have the identity \(\var(X) - \codim(X) = N + 1 - 3 c - i_X\) with \(c = \codim(X)\), hence by Corollary~\ref{corollary:WCI-fano-index} we can write
  \[
    \var(X) - \codim(X) = \sum_{j = 1}^c (d_j - 3) - \sum_{i = 0}^N (a_i - 1).
  \]
 
  By assumption we have \(d_j > 2\) for all \(j = 1, \ldots, c\), and \(a_i \leqslant 2\) for all \(i = 0, \ldots, N\). In other words,
  \[
    \sum_{j = 1}^c (d_j - 3) \geqslant 0, \quad \sum_{i = 0}^N (a_i - 1) = l, \quad
    l = \vert \{i \mid a_i = 2\} \vert.
  \]

  The pair \((\rho; \mu)\) is regular by Corollary~\ref{corollary:WCI-smooth-WF-combinatorics}, hence there exist at least \(l\) even numbers among \(d_j\). We have \(\sum_{j = 1}^c (d_j - 3) \geqslant l\), which is equivalent to \(\var(X) \geqslant \codim(X)\).
\end{proof}

\begin{lemma}\label{lemma:WCI-fano-s2-bound}
  Let \(X \subset \mathbb{P}(a_0, \ldots, a_N)\) be a smooth well-formed Fano weighted complete intersection which is not an intersection with a linear cone. Assume that \(\max_i(a_i) > 1\), and \(\dim(\vert \mathcal{O}_X(1) \vert) \geqslant \dim(X)\). The following estimate holds: \(s_2(X) \leqslant 3 \var(X) - 2\).
\end{lemma}

\begin{proof}
  Note that the assumption \(\max_i(a_i) > 1\) implies \(\var(X) > 0\) by Theorem~\ref{theorem:WCI-fano-small-variance}. Then we can assume that \(\codim(X) > \var(X)\). Actually, otherwise the result would be trivial:
  \[
    s_2(X) \leqslant \codim(X) \leqslant \var(X) = \min(\var(X), 3 \var (X) - 2).
  \]

  Throughout the proof we use the notation of Lemma~\ref{lemma:WCI-fano-variance-partition}.

  Note that by assumption and Lemma~\ref{lemma:WCI-QS-WF-O1-dimension} we have \(a_{\dim(X)} = 1\). Let \((d_1, \ldots, d_c)\) be multidegree of \(X\). We can also assume that \(\alpha_c > 0\). Actually, by Lemma~\ref{lemma:WCI-QS-WF-degrees-bound} we have \(d_c \geqslant 2 a_N\), i.e., \(a_N \leqslant a_{\dim(X)} + 1 + \alpha_c\). Then \(\alpha_c = 0\) implies that \(a_N = 2\), so in this case we can apply Lemma~\ref{lemma:WCI-fano-s2-bound-exception}. From Lemma~\ref{lemma:WCI-fano-variance-partition} we obtain
  \[
    c - s_2(X) = \vert \{j \mid d_j = 2 \} \vert =
    \vert \{j \mid a_{\dim(X) + j} = 1, \; \alpha_j = 0\} \vert.
  \]
  Put \(\mu = \vert \{j \mid \alpha_j = 0\} \vert\) and \(\nu = \vert \{j \mid a_{\dim(X) + j} > 1, \; \alpha_j = 0\} \vert\). Then we can write \(s_2(X) = c - \mu + \nu\).

  Let us first establish a lower bound on \(\mu\). As we already mentioned, \(a_{\dim(X)} = 1\), and \(\alpha_c > 0\). From Remark~\ref{remark:WCI-fano-variance-partition} and the Pigeonhole principle we obtain that
  \[
    \sum_{j = 1}^{c - 1} \alpha_j = \var(X) - 1 - \alpha_c, \quad \alpha_j \geqslant 0, \quad \alpha_c > 0; \quad
    \mu \geqslant c - 1 - \sum_{j = 1}^{c - 1} \alpha_j = c - \var(X) + \alpha_c.
  \]
  
  Now we can establish an upper bound on \(\nu\). By Lemma~\ref{lemma:WCI-fano-O1-bound-variance} we have \(\dim(\vert \mathcal{O}_X(1) \vert) \geqslant N - 2 \var(X)\), which is equivalent to the estimate \(2 \var(X) \geqslant \vert \{i \mid a_i > 1\} \vert\) by Lemma~\ref{lemma:WCI-QS-WF-O1-dimension}. Recall that \(\alpha_c > 0\), then it is clear that \(\vert \{i \mid a_i > 1\} \vert > \nu\). As a consequence, we obtain that \(\nu \leqslant 2 \var(X) - 1\).
\end{proof}

\begin{corollary}\label{corollary:WCI-fano-s2-bound}
  Let \(X \subset \mathbb{P}(\rho)\) be a smooth well-formed Fano weighted complete intersection which is not an intersection with a linear cone. Assume that \(\var(X) > 0\). The following estimate holds: \(s_2(X) \leqslant 3 \var(X) - 2\).
\end{corollary}

\begin{proof}
  If we assume \(\max(\rho) \leqslant 2\), then the statement follows from Corollary~\ref{corollary:CI-fano-s2-bound} and Lemma~\ref{lemma:WCI-fano-s2-bound-exception}. If we have \(\dim(\vert \mathcal{O}_X(1) \vert) < \dim(X)\), then by Lemma~\ref{lemma:WCI-fano-O1-bound-variance} we obtain the bound \(\codim(X) < 2 \var(X)\). At last, if we assume \(\dim(\vert \mathcal{O}_X(1) \vert) \geqslant \dim(X)\), and \(\max(\rho) > 1\), then we can use Lemma~\ref{lemma:WCI-fano-s2-bound}.
\end{proof}

\subsection{Bounds on the number of series and semi-series of given variance}\label{subsection:series-semiseries-bounds}

Finally, we prove that the number of series and semi-series can be explicitly bounded in terms of variance.

\begin{lemma}\label{lemma:CI-fano-series-cardinality}
  For any \(r \in \mathbb{Z}_{\geqslant 0}\) we have \(\vert \mathbb{S}^{\alpha}_r \vert = p(r)\), where \(p(r)\) is the number of integer partitions of \(r\).
\end{lemma}

\begin{proof}
  Follows from Lemma~\ref{lemma:CI-fano-variance-partition}.
\end{proof}

\begin{lemma}\label{corollary:WCI-fano-large-O1-variance-bound}
  Let \(X \subset \mathbb{P}(a_0, \ldots, a_N)\) be a smooth well-formed Fano weighted complete intersection of multidegree \((d_1, \ldots, d_c)\) which is not an intersection with a linear cone. Assume that \(\dim(\vert \mathcal{O}_X(1) \vert) \geqslant \dim(X)\). The following estimates hold: \(\max_i (a_i) \leqslant \var(X) + 1\), and \(\max_j (d_j) \leqslant 2 (\var(X) + 1)\).
\end{lemma}

\begin{proof}
  We may assume that \(\max(a_i) > 1\) (see Lemma~\ref{lemma:CI-fano-variance-partition}). From assumption and Lemma~\ref{lemma:WCI-QS-WF-O1-dimension} we obtain that \(a_{\dim(X)} = 1\). By Lemma~\ref{lemma:WCI-fano-variance-partition} we have \(d_c = a_N + a_{\dim(X)} + 1 + \alpha_c\), where \(0 \leqslant \alpha_c \leqslant \var(X) - 1\). Lemma~\ref{lemma:WCI-QS-WF-degrees-bound} implies that \(d_c \geqslant 2 a_N\), i.e., \(a_N \leqslant a_{\dim(X)} + 1 + \alpha_c\). Then we obtain the estimates
  \[
    a_N \leqslant a_{\dim(X)} + 1 + \alpha_c \leqslant \var(X) + 1, \quad
    d_c = a_N + a_{\dim(X)} + 1 + \alpha_c \leqslant 2 (\var(X) + 1). \qedhere
  \]
\end{proof}

\begin{notation}
  Let us denote by \(\mathbb{S}^{\beta}_r(c) \subset \mathbb{S}^{\beta}_r\) (respectively, by \(\mathbb{S}^{\gamma}_r(c) \subset \mathbb{S}^{\gamma}_r\)) the subset of series (respectively, of semi-series) of smooth well-formed Fano weighted complete intersections such that for any series \(\mathfrak{S} \in \mathbb{S}^{\beta}_r(c)\) (respectively, for any semi-series \(\mathfrak{S} \in \mathbb{S}^{\gamma}_r(c)\)) generated by a family \(\mathcal{F} \in \Theta\) we have \(\var(X) = r\) and \(\codim(X) = c\) for any smooth well-formed Fano weighted complete intersection \(X \in \mathcal{F}\).
\end{notation}

\begin{lemma}\label{lemma:WCI-fano-series-codim-bound}
  Put \(F_w(m) = \binom{m + w - 1}{m}\). The following estimate holds: \(\vert \mathbb{S}^{\beta}_r(c) \vert \leqslant F_r(c - 1) F_r(c)\).
\end{lemma}

\begin{proof}
  Let \(\mathfrak{S}(\mathcal{F}^{\rho}_{\mu}) \in \mathbb{S}^{\beta}_r(c)\) be a series, and \(X \in \mathcal{F}^{\rho}_{\mu}\) be a smooth well-formed Fano weighted complete intersection. Put \(\rho = (a_0, \ldots, a_N)\) and \(\mu = (d_1, \ldots, d_c)\). We have the following restrictions:
  \begin{itemize}
  \item \(a_i \neq 2\), and \(\vert \{i \mid a_i > 1 \} \vert \leqslant \codim(X) = c\) by Lemma~\ref{lemma:WCI-QS-WF-O1-dimension} and Remark~\ref{remark:WCI-amenable-O1-freeness};
  \item \(\sum a_i - \sum d_j = 1\), and \(s_2(X) = \codim(X) = c\) by Lemma~\ref{lemma:WCI-fano-generator};
  \item \(\max(a_0, \ldots, a_N) \leqslant r + 1\) by Lemma~\ref{corollary:WCI-fano-large-O1-variance-bound}.
  \end{itemize}

  Let us identify the family \(\mathcal{F}^{\rho}_{\mu}\) with the pair \((\rho; \mu)\). By assumption, Lemma~\ref{lemma:WCI-fano-variance-partition}, and Remark~\ref{remark:WCI-fano-variance-partition} we can uniquely represent the pair \((\rho; \mu)\) by tuples \((a_{\dim(X) + 1}, \ldots, a_N) \in \mathbb{Z}_{> 0}^c\) and \((\alpha_1, \ldots, \alpha_c) \in \mathbb{Z}_{\geqslant 0}^c\), where
  \begin{gather*}
    \alpha_j = d_j - a_{\dim(X) + j} - 1, \quad j = 1, \ldots, c - 1; \quad
    \alpha_c = d_c - a_N - 2, \\ \sum_{j = 1}^c \alpha_j = r - 1, \quad
    \dim(X) = \sum_{j = 1}^c (d_j - a_{\dim(X) + j}).
  \end{gather*}

  The number of possible tuples \((a_{\dim(X) + 1}, \ldots, a_N)\) is bounded by the number of non-decreasing sequences of numbers \(\{1, 3, \ldots, r + 1\}\) of length \(c\), which is equal to \(F_r(c)\). Moreover, the number of possible tuples \((\alpha_1, \ldots, \alpha_c)\) is bounded by the number of weak compositions of \(r - 1\) into \(c\) parts, which is equal to \(F_r(c - 1)\). Consequently, we obtain the required estimate \(\vert \mathbb{S}^{\beta}_r(c) \vert \leqslant F_r(c) F_r(c - 1)\).
\end{proof}

\begin{lemma}\label{lemma:WCI-fano-semiseries-codim-bound}
  Put \(F_w(m) = \binom{m + w - 1}{m}\). The following estimate holds:
  \(
  \vert \mathbb{S}_r^{\gamma}(c) \vert \leqslant F_{r + 2 c}(r + c) F_c((r - 1) (r + 2 c)).
  \)
\end{lemma}

\begin{proof}
  Let \(\mathfrak{S}(\mathcal{F}^{\rho}_{\mu}) \in \mathbb{S}^{\gamma}_r(c)\) be a semi-series, and \(X \in \mathcal{F}^{\rho}_{\mu}\) be a smooth well-formed Fano weighted complete intersection. Put \(\rho = (a_0, \ldots, a_N)\) and \(\mu = (d_1, \ldots, d_c)\). We have the following restrictions:
  \begin{itemize}
  \item \(\sum a_i - \sum d_j = 1\), and \(s_2(X) = \codim(X) = c\) by Lemma~\ref{lemma:WCI-fano-generator};
  \item \(\vert \{i \mid a_i > 1\} \leqslant \dim(X) = c + r\) by Lemma~\ref{lemma:WCI-QS-WF-O1-dimension} and Proposition~\ref{proposition:WCI-fano-CY-O1-dim-bound};
  \item \(\max(a_0, \ldots, a_N) \leqslant N = c + 2 r\) by~\cite[Theorem~1.1]{przyjalkowski/bounds}.
  \end{itemize}

  Let us identify the family \(\mathcal{F}^{\rho}_{\mu}\) with the pair \((\rho; \mu)\). Note that \((a_0, \ldots, a_N) \in \mathbb{Z}_{> 0}^{N + 1}\) is uniquely represented by the tuple \((a_{c + 1}, \ldots, a_N) \in \mathbb{Z}_{> 0}^{\dim(X)}\). The number of possible tuples \((a_{c + 1}, \ldots, a_N)\) is bounded by the number of non-decreasing sequences of numbers \(\{1, 2, \ldots, N\}\) of length \(r + c\), which is equal to \(F_{c + 2 r}(r + c)\). Moreover, for any fixed tuple \((a_{c + 1}, \ldots, a_N)\) by Lemma~\ref{lemma:WCI-fano-variance-partition} we can uniquely represent any multidegree \((d_1, \ldots, d_c)\) by the tuple \((\alpha_1, \ldots, \alpha_c) \in \mathbb{Z}_{\geqslant 0}^c\), where
  \[
    \alpha_j = d_j - a_{\dim(X) + j} - 1, \quad \alpha_c = d_c - a_N - a_{\dim(X)} - 1, \quad
    \sum_{j = 1}^c \alpha_j = \sum_{i = c + 1}^{\dim(X) - 1} a_i.
  \]
  Consequently, for any fixed tuple \((a_{c + 1}, \ldots, a_N)\) the number of possible multidegrees \((d_1, \ldots, d_c)\) is bounded by the number of weak compositions of \(\beta = \sum_{i = c + 1}^{\dim(X) - 1} a_i\) into \(c\) parts, which is equal to \(F_c(\beta)\). Note that we have \(\beta \leqslant (r - 1) N = (r - 1) (c + 2 r)\), so we can bound this number by \(F_c((r - 1) (c + 2 r))\). Consequently, we obtain the required estimate
  \(
  \vert \mathbb{S}^{\gamma}_r(c) \vert \leqslant F_{r + 2 c}(r + c) F_c((r - 1) (c + 2 r))
  \).
\end{proof}

\begin{corollary}\label{corollary:WCI-fano-series-semiseries-bound}
  Put \(F_w(m) = \binom{m + w - 1}{m}\). The following estimate hold:
  \[
    \vert \mathbb{S}^{\beta}_r \vert \leqslant \sum_{c = 1}^{3 r - 2} F_r(c - 1) F_r(c), \quad
    \vert \mathbb{S}^{\gamma}_r \vert \leqslant
    \sum_{c = 1}^{3 r - 2} F_{r + 2 c}(r + c) F_c((r - 1) (r + 2 c)).
  \]
\end{corollary}

\begin{proof}
  By Lemma~\ref{lemma:WCI-fano-generator} and Corollary~\ref{corollary:WCI-fano-s2-bound} we can write
  \(
  \mathbb{S}^{\beta}_r = \bigsqcup_{c = 1}^{3 r - 2} \mathbb{S}^{\beta}_r(c)
  \)
  and
  \(
  \mathbb{S}^{\gamma}_r = \bigsqcup_{c = 1}^{3 r - 2} \mathbb{S}^{\gamma}_r(c)
  \).
  Then the statement follows from Lemmas~\ref{lemma:WCI-fano-series-codim-bound} and~\ref{lemma:WCI-fano-semiseries-codim-bound}.
\end{proof} 

\section{Smooth well-formed Fano weighted complete intersections of anticanonical degree 1}\label{section:degree-one}

In this section we prove that smooth well-formed Fano weighted complete intersections of anticanonical degree one are non-amenable in the sense of Definition~\ref{definition:WCI-amenable}. Here we use definitions and notations of Section~\ref{section:preliminaries}.

\begin{notation}
  Let \(\rho = (a_0, \ldots, a_N)\) and \(\mu = (d_1, \ldots, d_m)\) be tuples of positive integers, and \(\nu_p\) be the \(p\)-adic valuation. For any prime \(p\) and \(n \in \mathbb{Z}_{\geqslant 0}\) we put \(I_p^n(\rho; \mu) = \{i \mid \nu_p(a_i) = n\}\) and \(J_p^n(\rho; \mu) = \{j \mid \nu_p(d_j) = n\}\).
\end{notation}

\begin{lemma}\label{lemma:regular-pair-degree-one}
  Let \(\rho = (a_0, \ldots, a_N)\) and \(\mu = (d_1, \ldots, d_m)\) be tuples of positive integers such that \((\rho; \mu)\) is a regular pair. The following assumptions are equivalent:
  \begin{itemize}
  \item \(\prod a_i = \prod d_j\);
  \item \(\vert I_p^n(\rho; \mu) \vert = \vert J_p^n(\rho; \mu) \vert\) for any prime \(p\) and \(n \in \mathbb{Z}_{> 0}\).
  \end{itemize}
\end{lemma}

\begin{proof}
  Assume that \(\prod a_i = \prod d_j\). The pair \((\rho; \mu)\) is regular, hence for any fixed prime \(p\) and \(n \in \mathbb{Z}_{> 0}\) we have the estimate \(\vert \{i \mid \nu_p(a_i) \geqslant n\} \vert \leqslant \vert \{j \mid \nu_p(d_j) \geqslant n\} \vert\), which is equivalent to the estimate
  \begin{equation}\label{equation:bijection-a}
    \sum_{s \geqslant n} \vert I_p^s(\rho; \mu) \vert \leqslant \sum_{s \geqslant n} \vert J_p^s(\rho; \mu) \vert.
  \end{equation}

  Put \(h_p = \max_s(\vert I_p^s(\rho; \mu) \vert > 0)\). Then we can rewrite (\ref{equation:bijection-a}) in the form
  \begin{equation}\label{equation:bijection-b}
    \sum_{s = n}^{h_p}
    \left (
      \vert J_p^s(\rho; \mu) \vert - \vert I_p^s(\rho; \mu) \vert
    \right ) +
    \sum_{s > h_p} \vert J_p^s(\rho; \mu) \vert \geqslant 0.
  \end{equation}
  
  Let us sum (\ref{equation:bijection-b}) over the free parameter \(n = 1, \ldots, h_p\). We obtain
  \begin{equation}\label{equation:bijection-c}
    \sum_{s = 1}^{h_p} s \left ( \vert J_p^s(\rho; \mu) \vert -
      \vert I_p^s(\rho; \mu) \vert \right ) +
    h_p \sum_{s > h_p} \vert J_p^s(\rho; \mu) \vert \geqslant 0.
  \end{equation}
  Note that \(\prod a_i = \prod d_j\) is equivalent to the identities
  \begin{equation}\label{equation:bijection-d}
    \sum_i \nu_p(a_i) = \sum_{s \geqslant 1} s \vert I_p^s(\rho; \mu) \vert =
    \sum_{s \geqslant 1} s \vert J_p^s(\rho; \mu) \vert = \sum_j \nu_p(d_j)
  \end{equation}
  for any prime \(p\). We can rewrite (\ref{equation:bijection-d}) as follows:
  \begin{equation}\label{equation:bijection-e}
    \left (
      \sum_{s = 1}^{h_p} s
      \left (
        \vert J_p^s(\rho; \mu) \vert - \vert I_p^s(\rho; \mu) \vert
      \right ) +
      h_p \sum_{s > h_p} \vert J_p^s(\rho; \mu) \vert
    \right ) +
    \sum_{s > h_p} (s - h_p) \vert J_p^s(\rho; \mu) \vert = 0.
  \end{equation}
  By comparing (\ref{equation:bijection-c}) and (\ref{equation:bijection-e}) we obtain that \(\vert J_p^s(\rho; \mu) \vert = 0\) for any \(s > h_p\). Then from (\ref{equation:bijection-b}) we derive that \(\vert I_p^s(\rho; \mu) \vert = \vert J_p^s(\rho; \mu) \vert\) for any \(s = 1, \ldots, h_p\).
  
  Now let us assume that \(\vert I_p^s(\rho; \mu) \vert = \vert J_p^s(\rho; \mu) \vert\) for any prime \(p\) and \(s \in \mathbb{Z}_{> 0}\). As we mentioned earlier, the identity \(\prod a_i = \prod d_j\) is equivalent to the identities (\ref{equation:bijection-d}), which completes the proof.
\end{proof}

\begin{notation}
  Denote by \(\mathcal{P}\) the set of regular pairs \((\rho; \mu) = (a_0, \ldots, a_N; d_1, \ldots, d_m)\) such that \(\prod a_i = \prod d_j\), and for any prime \(p\) and \(n \in \mathbb{Z}_{> 0}\) the subset \(I_p^n(\rho; \mu)\) is of type \((Q_1^+)\).
\end{notation}

\begin{lemma}\label{lemma:p-class-not-prime-power}
  Let \(\rho = (a_0, \ldots, a_N)\) and \(\mu = (d_1, \ldots, d_m)\) be tuples of positive integers. Assume that \((\rho; \mu) \in \mathcal{P}\), and \(a_i \neq d_j\) for any \(i = 0, \ldots, N\) and \(j = 1, \ldots, m\). Then \(d_j\) is not a prime power for any \(j = 1, \ldots, m\).
\end{lemma}

\begin{proof}
  Let us assume that \(d_j = p^n\) for some prime \(p\) and \(n \in \mathbb{Z}_{> 0}\). Then by Lemma~\ref{lemma:regular-pair-degree-one} and the definition of the set \(\mathcal{P}\) we would have \(a_i = p^n\) for some \(i = 0, \ldots, N\), so we obtain a contradiction.
\end{proof}

\begin{notation}
  Let \(\rho = (a_0, \ldots, a_N)\) and \(\mu = (d_1, \ldots, d_m)\) be tuples of positive integers such that the pair \((\rho; \mu)\) is regular, and \(p\) be a prime number. Then the pair \((\rho^p; \mu^p)\) is regular by~\cite[Lemma~4.5]{pizzato/nonvanishing}, where
  \[
    d_j^p = 
    \begin{cases}
      d_j / p & \text{if} \; p \mid d_j, \\
      d_j & \text{if} \; p \nmid d_j;
    \end{cases}
    \quad
    a_i^p = 
    \begin{cases}
      a_i / p & \text{if} \; p \mid a_i, \\
      a_i & \text{if} \; p \nmid a_i.
    \end{cases}
  \]
\end{notation}

\begin{notation}
  Let \(\rho = (a_0, \ldots, a_N)\) and \(\mu = (d_1, \ldots, d_m)\) be tuples of positive integers such that the pair \((\rho; \mu)\) is regular. We denote by \((\widetilde{\rho}; \widetilde{\mu})\) the unique regular pair obtained by removing all \(d_j\) and \(a_i\) such that \(d_j = a_i\) (see~\cite[Lemma~4.4]{pizzato/nonvanishing}).
\end{notation}

\begin{lemma}
  Let \(\rho = (a_0, \ldots, a_N)\) and \(\mu = (d_1, \ldots, d_m)\) be tuples of positive integers such that \((\rho; \mu)\) is a regular pair. Then \((\widetilde{\rho^p}; \widetilde{\mu^p}) \in \mathcal{P}\) for any prime number \(p\).
\end{lemma}

\begin{proof}
  The first property in the definition of the set \(\mathcal{P}\) is preserved by Lemma~\ref{lemma:regular-pair-degree-one}. The second property is also preserved, since for any prime number \(p\) and \(n \in \mathbb{Z}_{> 0}\) we have the identities
  \[
    I_p^n (\widetilde{\rho^p}; \widetilde{\mu^p}) =
    I_p^{n + 1} (\widetilde{\rho}; \widetilde{\mu}), \quad
    J_p^n (\widetilde{\rho^p}; \widetilde{\mu^p}) =
    J_p^{n + 1} (\widetilde{\rho}; \widetilde{\mu}). \qedhere
  \]
\end{proof}

\begin{lemma}\label{lemma:regular-pair-degree-one-bound}
  Let \(\rho = (a_0, \ldots, a_N)\) and \(\mu = (d_1, \ldots, d_m)\) be tuples of positive integers such that \((\rho; \mu) \in \mathcal{P}\) is a regular pair. Then the following estimate holds:
  \(
  \vert \{i \mid a_i > 1\} \vert \geqslant \vert \{j \mid d_j > 1\} \vert
  \),
  where we have the equality if and only if \((\widetilde{\rho}; \widetilde{\mu})\) is the empty pair.
\end{lemma}

\begin{proof}
  We prove the statement by induction on the number \(n\) of distinct prime divisors of \(\prod d_j = \prod a_i\). Put
  \[
    I = \{i \mid a_i > 1\} = \bigcup_{t = 1}^n I_{(p_t)}, \;
    J = \{j \mid d_j > 1\} = \bigcup_{t = 1}^n J_{(p_t)}; \;
    I_{(p_t)} = \{i \mid \nu_{p_t}(a_i) > 0\}, \;
    J_{(p_t)} = \{j \mid \nu_{p_t}(d_j) > 0\}.
  \]
  Throughout the proof we implicitly replace any regular pair \((\rho'; \mu')\) with \((\widetilde{\rho'}; \widetilde{\mu'})\) whenever it is convenient.

  The basis of induction is trivial: if \(n = 0\), then the pair \((\rho; \mu)\) is empty.

  Now let us assume that the statement holds for any \(k < n\). Put
  \[
    \widetilde{I} = \bigcup_{t = 1}^{n - 1} I_{(p_t)}, \quad
    \widetilde{J} = \bigcup_{t = 1}^{n - 1} J_{(p_t)}; \quad
    I^{\pr}_p = \{i \mid a_i = p^l, \; l \in \mathbb{Z}_{> 0}\}, \quad
    J^{\pr}_p = \{j \mid d_j = p^l, \; l \in \mathbb{Z}_{> 0}\}.
  \]
  Then we can write \(\vert I \vert = \vert \widetilde{I} \vert + \vert I^{\pr}_{p_n} \vert\) and \(\vert J \vert = \vert \widetilde{J} \vert + \vert J^{\pr}_{p_n} \vert\). Note that by Lemma~\ref{lemma:p-class-not-prime-power} we have \(J^{\pr}_{p_n} = \emptyset\).

  Let us replace the pair \((\rho; \mu)\) with the pair \((\rho^{p_n}; \mu^{p_n})\) at least \(\max_s(\vert I_p^s(\rho; \mu) \vert > 0)\) times. We denote the obtained pair by \((\rho^{p_n}_*; \mu^{p_n}_*)\). Now we can apply the induction hypothesis to \((\rho^{p_n}_*; \mu^{p_n}_*)\), hence \(\vert \widetilde{I} \vert \geqslant \vert \widetilde{J} \vert\). We have the estimate \(\vert I \vert - \vert J \vert = \vert \widetilde{I} \vert - \vert \widetilde{J} \vert + \vert I^{\pr}_{p_n} \vert \geqslant 0\). Moreover, the equality holds if and only if the pair \((\rho^{p_n}_*; \mu^{p_n}_*)\) is empty, and \(\vert I^{\pr}_{p_n} \vert = 0\). In other words, the pair \((\widetilde{\rho}; \widetilde{\mu})\) should be empty as well.
\end{proof}

\begin{proposition}\label{proposition:WCI-fano-degree-one}
  Let \(X \subset \mathbb{P}(\rho)\) be a smooth well-formed Fano weighted complete intersection of multidegree \((d_1, \ldots, d_c)\) which is not an intersection with a linear cone. Assume that \((-K_X)^{\dim(X)} = 1\).
  \begin{enumerate}
  \item The following estimate holds: \(\dim(\vert \mathcal{O}_X(1) \vert) < \dim(X)\).
  \item The degree \(d_j\) is not a prime power for all \(j = 1, \ldots, c\).
  \end{enumerate}
\end{proposition}

\begin{proof}
  Put \(\rho = (a_0, \ldots, a_N)\) and \(\mu = (d_1, \ldots, d_c)\). On the one hand, the pair \((\rho; \mu)\) is consistent and strictly regular by Proposition~\ref{corollary:WCI-smooth-WF-combinatorics}. On the other hand, by Corollary~\ref{corollary:WCI-fano-degree-one} we have the identity \(\prod d_j = \prod a_i\). Consequently, the pair \((\rho; \mu)\) lies in the set \(\mathcal{P}\). Note that \(\vert \{j \mid d_j > 1\} \vert = \codim(X)\) by assumption. Then Lemma~\ref{lemma:regular-pair-degree-one-bound} implies the bound \(\vert \{i \mid a_i > 1\} \vert > \codim(X)\), which is equivalent to first statement by Lemma~\ref{lemma:WCI-QS-WF-O1-dimension}. The second statement follows from Lemma~\ref{lemma:p-class-not-prime-power}.
\end{proof}

\begin{corollary}\label{corollary:WCI-fano-degree-one-finite}
  The number of families \(\mathcal{F} \in \Theta\) such that there exists a smooth well-formed Fano weighted complete intersection \(X \in \mathcal{F}\) of anticanonical degree one and given variance is finite. 
\end{corollary}

\begin{proof}
  Let \(\mathcal{F} \in \Theta\) be a family such that there exists a smooth well-formed Fano weighted complete intersection \(X \in \mathcal{F}\) with \((-K_X)^{\dim(X)} = 1\). Theorem~\ref{theorem:WCI-fano-degree-one} implies the bound \(\dim(\vert \mathcal{O}_X(1) \vert) < \dim(X)\), hence \(X\) is not amenable by Remark~\ref{remark:WCI-amenable-O1-freeness}. Note that \(i_X = 1\). Propositions~\ref{proposition:WCI-fano-indexes},~\ref{proposition:WCI-fano-structure}, and Remark~\ref{remark:WCI-fano-index} imply that \(\mathcal{F}\) can be identified with a generating family of a semi-series in \(\Theta\). Their number is finite by Theorem~\ref{theorem:WCI-fano-series-semiseries-bound}.
\end{proof}

\appendix

\section{Regular sequences of weighted homogeneous polynomials are generic}\label{section:regular-sequences}

Here we prove that regular sequences of weighted homogeneous polynomials of given weights and degrees form a Zariski-open subset in the space of all sequences of such weighted homogeneous polynomials.

\begin{remark}
  Throughout this section we continue to follow Notations~\ref{notation:polynomial-ring} and~\ref{notation:polynomial-sequences}.
\end{remark}

\begin{definition}
  For any tuples \(\rho\) and \(\mu\) of positive integers we refer to the map
  \[
    \mathcal{H}^{\rho}_{\mu} \colon R^{\rho}_{\mu} \rightarrow \mathbb{Z}[[T]], \quad
    F \mapsto \HS_{R^{\rho} / \langle F \rangle}(T),
  \]
  as the \emph{Hilbert series evaluation map}.
\end{definition}

\begin{notation}
  Let \(S', S'' \in \mathbb{Z}[[T]]\) be formal series. We denote by \(S' \preccurlyeq S''\) the coefficient-wise inequality.
\end{notation}

We are going to prove that for any tuples \(\rho\) and \(\mu\) of positive integers the map \(\mathcal{H}^{\rho}_{\mu}\) is upper-semicontinuous: we claim that the subset
\(
\{F \in R^{\rho}_{\mu} \mid \mathcal{H}^{\rho}_{\mu}(F) \preccurlyeq S\} \subset R^{\rho}_{\mu}
\)
is Zariski-open for any formal series \(S \in \im(\mathcal{H}^{\rho}_{\mu})\). As a consequence, we will obtain that regular sequences in \(R^{\rho}_{\mu}\) form a Zariski-open subset (possibly, empty).

\begin{example}
  By definition \(\HS_{R^{\rho}}(T)\) is the generating series of denumerants \(D(\rho; \bullet)(T)\) (see Lemma~\ref{lemma:denumerants-series}).
\end{example}

\begin{lemma}[{\cite[Theorem~8.20]{miller/algebra}}]\label{lemma:K-polynomial}
  Let \(\rho = (a_0, \ldots, a_N)\) be a tuple of positive integers, and \(F \in R^{\rho}_{\mu}\) be a sequence of weighted homogeneous polynomials. Then the Hilbert series of \(R^{\rho} / \langle F \rangle\) has the following form:
  \[
    \HS_{R^{\rho} / \langle F \rangle}(T) = \frac{N(T)}{\prod_{i = 0}^N (1 - T^{a_i})}, \quad
    N(T) \in \mathbb{Z}[T].
  \]
\end{lemma}

\begin{definition}[{\cite[Definition~8.21]{miller/algebra}}]
  Let \(F \in R^{\rho}_{\mu}\) be a sequence of weighted homogeneous polynomials. The polynomial \(N(t)\) from Lemma~\ref{lemma:K-polynomial} is called the \emph{K-polynomial} of \(R^{\rho} / \langle F \rangle\), which we denote by \(\mathcal{K}(R^{\rho} / \langle F \rangle)\).
\end{definition}

The degree of the K-polynomial \(\mathcal{K}(R^{\rho} / \langle F \rangle)\) is bounded for any sequence \(F \in R^{\rho}_{\mu}\).

\begin{lemma}\label{lemma:K-polynomial-degree-bound}
  Let \(\rho\) and \(\mu\) be tuples of positive integers. There exists a number \(C(\rho; \mu) \in \mathbb{Z}_{\geqslant 0}\) such that \(\deg (\mathcal{K}(R^{\rho} / \langle F \rangle)) \leqslant C(\rho; \mu)\) for any sequence \(F \in R^{\rho}_{\mu}\) of weighted homogeneous polynomials.
\end{lemma}

\begin{proof}
  Let \(F \in R^{\rho}_{\mu}\) be a sequence of weighted homogeneous polynomial. The K-polynomial of \(R^{\rho} / \langle F \rangle\) records the alternating sum of its graded Betti numbers (see~\cite[Proposition~8.23]{miller/algebra}). Moreover, by~\cite[Theorem~8.29]{miller/algebra} in order to bound the degree of K-polynomial we can always assume that \(F\) is a monomial sequence, and there is finite number of such sequences in \(R^{\rho}_{\mu}\).  
\end{proof}

Then we can effectively approximate a Hilbert series \(\mathcal{H}^{\rho}_{\mu}(F)\) by its appropriate truncation.

\begin{definition}\label{definition:hilbert-series-truncation}
  Let \(\rho\) and \(\mu\) be tuples of positive integers. Let \(C(\rho; \mu) \in \mathbb{Z}_{> 0}\) be the minimal number such that \(\deg (\mathcal{K}(R^{\rho} / \langle F \rangle)) \leqslant C(\rho; \mu)\) for any sequence \(F \in R^{\rho}_{\mu}\) of weighted homogeneous polynomials (which exists by Lemma~\ref{lemma:K-polynomial-degree-bound}). Then for any sequence \(F \in R^{\rho}_{\mu}\) of weighted homogeneous polynomials we define the \emph{truncation} of the associated Hilbert series \(\mathcal{H}^{\rho}_{\mu}(F) = \sum_{i = 0}^{\infty} s_i T^i\) as \(\widetilde{\mathcal{H}}^{\rho}_{\mu}(F) = \sum_{i = 0}^{C(\rho; \mu)} s_i T^i\).
\end{definition}

\begin{corollary}\label{corollary:hilbert-series-coincidence}
  Let \(\rho\) and \(\mu\) be tuples of positive integers. For any sequences \(F,G \in R^{\rho}_{\mu}\) the assumption
  \(
  \widetilde{\mathcal{H}}^{\rho}_{\mu}(F) =
  \widetilde{\mathcal{H}}^{\rho}_{\mu}(G)
  \)
  implies that \(\mathcal{H}^{\rho}_{\mu}(F) = \mathcal{H}^{\rho}_{\mu}(G)\).
\end{corollary}

\begin{proof}
  Lemma~\ref{lemma:K-polynomial} implies that there exist polynomials \(M(T),N(T) \in \mathbb{Z}[T]\) such that
  \[
    \HS_{R^{\rho} / \langle F \rangle}(T) = \frac{M(T)}{\prod_{i = 0}^N (1 - T^{a_i})}, \quad
    \HS_{R^{\rho} / \langle G \rangle}(T) = \frac{N(T)}{\prod_{i = 0}^N (1 - T^{a_i})}.
  \]
  Moreover, by Lemma~\ref{lemma:K-polynomial-degree-bound} there exists a number \(C(\rho; \mu) \in \mathbb{Z}_{\geqslant 0}\) such that we have \(\deg(M(T)) \leqslant C(\rho; \mu)\) and \(\deg(N(T)) \leqslant C(\rho; \mu)\). Then the statement follows from~\cite[Theorem~4.1.1]{stanley/combinatorics}.
\end{proof}

We can characterise the Hilbert series in terms of \emph{Macaulay matrices} (for example, see~\cite[\nopp 41.3]{mora/equation}).

\begin{definition}
  Let \(F \in R^{\rho}_{\mu}\) be a sequence of weighted homogeneous polynomials, and \(\preccurlyeq\) be a monomial ordering on \(R^{\rho}\). For any \(d \geqslant 0\) the \emph{Macaulay matrix} \(\Mac_{\preccurlyeq, d}(F)\) is the matrix over \(\Bbbk\)
  \begin{itemize}
  \item with \(D(\rho; d)\) columns, indexed by monomials of degree \(d\) and sorted by \emph{decreasing} ordering with respect to \(\preccurlyeq\);
  \item with \(\sum_{j = 1}^m D(\rho; d_j)\) rows, indexed by pairs \((j, \mu)\) for \(j = 1, \ldots, m\) and monomials \(\mu\) of degree \(d - d_j\); the indexes are sorted by \emph{increasing} \(j\) first, and then by \emph{decreasing} \(\mu\);
  \item such that the matrix coefficient \(\Mac_{\preccurlyeq, d}(F)_{(j, \mu), \mu'}\) is equal to the coefficient of the monomial \(\mu'\) in the polynomial \(\mu f_j\) for any monomials \(\mu\) and \(\mu'\) of degree \(d\) and \(d - d_j\), respectively.
  \end{itemize}
\end{definition}

\begin{definition}
  Let \(F \in R^{\rho}_{\mu}\) be a sequence of weighted homogeneous polynomials, and \(\preccurlyeq\) be any monomial ordering on \(R^{\rho}\). For any \(d \in \mathbb{Z}_{\geqslant 0}\) the \emph{corank} of the Macaulay matrix \(\Mac_{\preccurlyeq, d}(F)\) is defined as
  \(
  \corank(\Mac_d(F)) = D(\rho; d) - \rank(\Mac_{\preccurlyeq, d}(F))
  \).
\end{definition}

\begin{lemma}
  Let \(F \in R^{\rho}_{\mu}\) be a sequence of weighted homogeneous polynomials. Then the Hilbert series of \(R^{\rho} / \langle F \rangle\) coincides with the following generating series:
  \(
    \HS_{R^{\rho} / \langle F \rangle}(T) = \sum_{d = 0}^{\infty} \corank(\Mac_d(F)) T^d
  \).
\end{lemma}

\begin{proof}
  For any \(d \in \mathbb{Z}_{\geqslant 0}\) we can identify the \(d\)-graded component \((R^{\rho} / \langle F \rangle)_d\) of \(R^{\rho} / \langle F \rangle\) with the quotient space of the \(d\)-graded component \(R^{\rho}_d\) modulo the relations described by the Macaulay matrix \(\Mac_d(F)\). In other words, we have the identity
  \(
  \dim((R^{\rho} / \langle F \rangle)_d) = \dim(R^{\rho}_d) - \rank(\Mac_d(F)) =
  D(\rho; d) - \rank(\Mac_d(F))
  \).
\end{proof}

\begin{lemma}\label{lemma:hilbert-series-weak-semicontinuity}
  Let \(\rho\) and \(\mu\) be tuples of positive integers. The function
  \[
    \dim((R^{\rho} / \langle \bullet \rangle)_d) \colon R^{\rho}_{\mu} \rightarrow \mathbb{Z}_{\geqslant 0}, \quad
    F \mapsto \dim((R^{\rho} / \langle F \rangle)_d),
  \]
  is upper-semicontinuous for any \(d \in \mathbb{Z}_{\geqslant 0}\).
\end{lemma}

\begin{proof}
  By the definition of a Macaulay matrix we have the linear map
  \[
    \Mac_d \colon R^{\rho}_{\mu} \rightarrow \Mat_{\Bbbk}
    \left (
      \sum_{j = 1}^m D(\rho; d - d_j), D(\rho; d)
    \right ).
  \]
  The rank function is lower-semicontinuous on this matrix space, and hence it is lower-semicontinuous on the subspace \(\Mac_d(R^{\rho}_{\mu})\), and we are done.
\end{proof}

\begin{lemma}\label{lemma:hilbert-series-semicontinuity}
  Let \(\rho\) and \(\mu\) be tuples of positive integers. For any sequence \(G \in R^{\rho}_{\mu}\) of weighted homogeneous polynomials the following subset is Zariski-open: \(W^{\rho}_{\mu}(G) = \{F \in R^{\rho}_{\mu} \mid \mathcal{H}^{\rho}_{\mu}(F) \preccurlyeq \mathcal{H}^{\rho}_{\mu}(G)\}\).
\end{lemma}

\begin{remark}
  Note that we always have \(G \in W^{\rho}_{\mu}(G)\), so \(W^{\rho}_{\mu}(G)\) is not empty.
\end{remark}

\begin{proof}
   Lemma~\ref{lemma:hilbert-series-weak-semicontinuity} implies that \(W^{\rho}_{\mu}(G)\) is a countable intersection of Zariski-open subsets: 
  \[
    W^{\rho}_{\mu}(G) = \{F \in R^{\rho}_{\mu} \mid \widetilde{\mathcal{H}}^{\rho}_{\mu}(F)
    \preccurlyeq \widetilde{\mathcal{H}}^{\rho}_{\mu}(G)\} \cap
    \bigcap_{d = C(\rho; \mu) + 1}^{\infty} \{F \in R^{\rho}_{\mu} \mid
    \dim((R^{\rho} / \langle F \rangle)_d) \leqslant s_d \}, \quad
    \mathcal{H}^{\rho}_{\mu}(G) = \sum_{i = 0}^{\infty} s_i T^i.
  \]
  where \(\widetilde{\mathcal{H}}^{\rho}_{\mu}\) is the truncation (see Definition~\ref{definition:hilbert-series-truncation}). The first subset is Zariski-open and can be presented as
  \[
    \{F \in R^{\rho}_{\mu} \mid \widetilde{\mathcal{H}}^{\rho}_{\mu}(F) \preccurlyeq
    \widetilde{\mathcal{H}}^{\rho}_{\mu}(G)\} =
    \bigcup_{\widetilde{S} \in \widetilde{\Lambda}}
    \{F \in R^{\rho}_{\mu} \mid \widetilde{\mathcal{H}}^{\rho}_{\mu}(F)
    \preccurlyeq \widetilde{S}\}, \quad
    \widetilde{\Lambda} = \{\widetilde{\mathcal{H}}^{\rho}_{\mu}(F) \mid
    \widetilde{\mathcal{H}}^{\rho}_{\mu}(F) \preccurlyeq
    \widetilde{\mathcal{H}}^{\rho}_{\mu}(G)\}.
  \]
  Note that the set \(\widetilde{\Lambda}\) is finite by construction. Moreover, Corollary~\ref{corollary:hilbert-series-coincidence} implies that any element \(\widetilde{\mathcal{H}}^{\rho}_{\mu}(F) \in \widetilde{\Lambda}\) lifts to a unique Hilbert series \(\mathcal{H}^{\rho}_{\mu}(F)\). In other words, we can identify the finite set \(\widetilde{\Lambda}\) with the following set:
  \[
    \Lambda = \{\mathcal{H}^{\rho}_{\mu}(F) \mid \widetilde{\mathcal{H}}^{\rho}_{\mu}(F)
    \preccurlyeq \widetilde{\mathcal{H}}^{\rho}_{\mu}(G)\}, \quad
    \vert \Lambda \vert = \vert \widetilde{\Lambda} \vert.
  \]
 
  In other words, we can write \(\{F \in R^{\rho}_{\mu} \mid \widetilde{\mathcal{H}}^{\rho}_{\mu}(F) \preccurlyeq \widetilde{\mathcal{H}}^{\rho}_{\mu}(G)\}\) as a finite union of Zariski-open subsets:
  \[
    \{F \in R^{\rho}_{\mu} \mid \widetilde{\mathcal{H}}^{\rho}_{\mu}(F) \preccurlyeq
    \widetilde{\mathcal{H}}^{\rho}_{\mu}(G)\} =
    \bigcup_{S \in \Lambda} \{F \in R^{\rho}_{\mu} \mid \mathcal{H}^{\rho}_{\mu}(F) \preccurlyeq S\} =
    \bigcup_{\widetilde{S} \in \widetilde{\Lambda}}
    \{F \in R^{\rho}_{\mu} \mid \widetilde{\mathcal{H}}^{\rho}_{\mu}(F)
    \preccurlyeq \widetilde{S}\}.
  \]  
  As a consequence, we can present \(W^{\rho}_{\mu}(G)\) as a finite union of Zariski-open subsets as well:
  \[
    W^{\rho}_{\mu}(G) =
    \bigcup_{S \in \Lambda'} \{F \in R^{\rho}_{\mu} \mid \mathcal{H}^{\rho}_{\mu}(F) \preccurlyeq S\}, \quad
    \Lambda' = \{\mathcal{H}^{\rho}_{\mu}(F) \mid \mathcal{H}^{\rho}_{\mu}(F)
    \preccurlyeq \mathcal{H}^{\rho}_{\mu}(G)\} \subset \Lambda. \qedhere
  \]
\end{proof}

\begin{definition}
  Let \(\rho\) and \(\mu\) be tuples of positive integers, and \(M^{\rho}_{\mu} \in \im(\widetilde{\mathcal{H}}^{\rho}_{\mu})\) be the unique polynomial such that \(M^{\rho}_{\mu} \preccurlyeq \widetilde{\mathcal{H}}^{\rho}_{\mu}(F)\) for any \(F \in R^{\rho}_{\mu}\) (which exists by Lemma~\ref{lemma:hilbert-series-weak-semicontinuity}). Its extension \(\MHS^{\rho}_{\mu} \in \im(\mathcal{H}^{\rho}_{\mu})\) (which is unique by Corollary~\ref{corollary:hilbert-series-coincidence}) will be referred to as the \emph{minimal Hilbert series}.
\end{definition}

\begin{corollary}\label{corollary:minimal-hilbert-series-generic}
  Let \(\rho\) and \(\mu\) be tuples of positive integers. Then sequences of weighted homogeneous polynomials in \(R^{\rho}_{\mu}\) with the associated minimal Hilbert series
  \(
    \{F \in R^{\rho}_{\mu} \mid \HS_{R^{\rho} / \langle F \rangle} (T) = \MHS^{\rho}_{\mu}(T)\} \subset R^{\rho}_{\mu}
  \)
  form a non-empty Zariski-open subset of \(R^{\rho}_{\mu}\).
\end{corollary}

\begin{lemma}
  Let \(F \in R^{\rho}_{\mu}\) be a sequence of weighted homogeneous polynomials, and \(f \in R^{\rho} / \langle F \rangle\) be a homogeneous element of degree \(d\). The following inequality holds:
  \(
    \HS_{R^{\rho} / \langle F, f \rangle}(T) \succcurlyeq (1 - T^d) \HS_{R^{\rho} / \langle F \rangle}(T)
  \),
  where the equality holds if and only if \(f \in R^{\rho} / \langle F \rangle\) is not a zero divisor.
\end{lemma}

\begin{proof}
  We have the following exact sequence of graded \(R^{\rho}\)-modules: 
  \[
    0 \rightarrow K \rightarrow R^{\rho} / \langle F \rangle \xrightarrow{\cdot f} R^{\rho} / \langle F \rangle \rightarrow R^{\rho} / \langle F, f \rangle \rightarrow 0,
  \]
  where \(K\) is the kernel of the multiplication map. It induces exact sequences of homogeneous components
  \[
    0 \rightarrow K_u \rightarrow (R^{\rho} / \langle F \rangle)_u \xrightarrow{\cdot f} (R^{\rho} / \langle F \rangle)_{u + d} \rightarrow
    (R^{\rho} / \langle F, f \rangle)_{u + d} \rightarrow 0, \quad u \in \mathbb{Z}_{\geqslant 0}.
  \]
  Consequently, for any \(u \in \mathbb{Z}_{\geqslant 0}\) we have the identity
  \[
    \dim((R^{\rho} / \langle F \rangle)_{u + d}) - \dim((R^{\rho} / \langle F \rangle)_u) =
    \dim((R^{\rho} / \langle F, f \rangle)_{u + d}) - \dim(K_u).
  \]
  Let us multiply both sides by \(T^{d + u}\), and summarise over \(\mathbb{Z}_{\geqslant 0}\). We obtain that
  \[ 
    \HS_{R^{\rho} / \langle F, f \rangle}(T)  = (1 - T^d) \HS_{R^{\rho} / \langle F \rangle}(T) +
    \left ( \sum_{u = 0}^{\infty} \dim(K_u) T^u \right ) \cdot T^d. \qedhere
  \]
\end{proof}

\begin{corollary}\label{corollary:hilbert-series-regular-sequence}
  Let \(\rho = (a_0, \ldots, a_N)\) and \(\mu = (d_1, \ldots, d_m)\) be tuples of positive integers, where \(m \leqslant N + 1\), and \(F \in R^{\rho}_{\mu}\) be a sequence of weighted homogeneous polynomials. Then \(F\) is a regular sequence if and only
  \[
    \HS_{R^{\rho} / \langle F \rangle}(T) =
    \frac{\prod_{j = 1}^m (1 - T^{d_j})}{\prod_{i = 0}^N (1 - T^{a_i})} = \MHS^{\rho}_{\mu}(T).
  \]
\end{corollary}

\begin{corollary}\label{corollary:regular-sequences-generic}
  Regular sequences in \(R^{\rho}_{\mu}\) form a Zariski-open subset of \(R^{\rho}_{\mu}\) (possibly, empty).
\end{corollary}

\begin{proof}
  Follows from Corollaries~\ref{corollary:minimal-hilbert-series-generic} and~\ref{corollary:hilbert-series-regular-sequence}.
\end{proof}

\begin{corollary}\label{corollary:height-lower-bound-generic}
  Let \(\rho\) and \(\mu\) be tuples of positive integers. For any \(n \in \mathbb{Z}_{> 0}\) the following subset is Zariski-open: \(U^{\rho}_{\mu}(n) = \{F \in R^{\rho}_{\mu} \mid \height(\langle F \rangle) \geqslant n\}\). 
\end{corollary}

\begin{proof}
  Put \(\mu = (d_1, \ldots, d_m)\), and let \(F = (f_1, \ldots, f_m) \in R^{\rho}_{\mu}\) be any sequence. By~\cite[Proposition~1.5.11]{bruns/cohenmacaulay} we have \(\height(\langle F \rangle) \geqslant n\) if and only if there exist a tuple \(\tau = (b_1, \ldots, b_n)\) of positive integers and a regular sequence \(G = (g_1, \ldots, g_n) \in R^{\rho}_{\tau}\) such that \(g_k \in \langle F \rangle\) for any \(k = 1, \ldots, n\). In other words, there should exist an \((m, n)\)-matrix \(H\) such that its elements \(H_{j, k} \in R^{\rho}\) are homogeneous of degree \(b_k - d_j\), and we have \(H \cdot F = G\).

  For any tuple \(\tau\) let us denote by \(M^{\rho}_{\mu, \tau}\) the space of all \((m, n)\)-matrices \(H\) such that their elements \(H_{j, k} \in R^{\rho}\) are homogeneous of degree \(b_k - d_j\). We have the natural map \(\Phi^{\rho}_{\mu, \tau} \colon M^{\rho}_{\mu, \tau} \times R^{\rho}_{\mu} \rightarrow R^{\rho}_{\tau}\), \((H, F) \mapsto H \cdot F = G\). Let \(U^{\rho}_{\tau} \subset R^{\rho}_{\tau}\) be the subset of regular sequences, which is Zariski-open by Corollary~\ref{corollary:regular-sequences-generic}. Then \((\Phi^{\rho}_{\mu, \tau})^{-1}(U^{\rho}_{\tau})\) is Zariski-open in \(M^{\rho}_{\mu, \tau} \times R^{\rho}_{\mu}\). Since the projection \(\pi \colon M^{\rho}_{\mu, \tau} \times R^{\rho}_{\mu} \rightarrow R^{\rho}_{\mu}\) is an open map, the subset \(W^{\rho}_{\mu}(\tau) = \pi((\Phi^{\rho}_{\mu, \tau})^{-1}(U^{\rho}_{\tau}))\) is also Zariski-open. Then \(U^{\rho}_{\mu}(n) = \cup_{\tau} W^{\rho}_{\mu}(\tau)\) is Zariski-open, and we are done.
\end{proof}

\section{List of series and semi-series of small variance}\label{section:small-variance}

In this section we provide the list of all generating families \(\mathcal{F} \in \Theta\) of smooth well-formed Fano weighted complete intersections of variance up to 4.

The classification of smooth well-formed Fano weighted complete intersections of given variance can be reduced to brute force. Main restrictions on weights and degrees in terms of variance are described in the proof of Lemmas~\ref{lemma:WCI-fano-series-codim-bound} and~\ref{lemma:WCI-fano-semiseries-codim-bound}. We also use restrictions provided by Lemma~\ref{lemma:WCI-QS-WF-degrees-bound} and Corollaries~\ref{corollary:WCI-fano-index} and~\ref{corollary:WCI-smooth-WF-combinatorics}. We have written an implementation of this brute force algorithm in SAGE (for any variance), which can be found at \url{https://github.com/MikhailOvcharenko/fano-WCI}. Note that all these restrictions are only necessary, so we should also manually check that obtained candidates for generating families are actually families of smooth well-formed Fano weighted complete intersections (see Subsection~\ref{subsection:WCI-basics} for details).

The numeration of families here is lexicographic with respect to variance, dimension, anticanonical degree, dimension of the anticanonical linear system, weights, and degrees, respectively. Note that all families \(\mathcal{F}\) of weighted complete intersections listed below are non-empty (see Lemma~\ref{lemma:WCI-family-nonempty}), and their general elements \(X \in \mathcal{F}\) are smooth, well-formed and Fano (see Corollary~\ref{corollary:WCI-smooth-WF-criterion} and Proposition~\ref{proposition:WCI-QS-WF-adjunction}). The last column shows whether \(X\) is (strictly) amenable in the sense of Definition~\ref{definition:WCI-amenable}.

\begin{longtable}{ccccccc}
  \toprule
  No. & \(\dim(X)\) & \((-K_X)^{\dim(X)}\) & \(h^0(X, -K_X)\) & \(\mathbb{P}(\rho)\) & Degrees & Amenability \\ \midrule \endhead
  0.1 & \(0\) & \(-\) & \(-\) & \(\pt\) & \(-\) & Strict \\
  \midrule  
  1.1 & \(2\) & \(1\) & \(2\) & \(\mathbb{P}(1^{(2)}, 2, 3)\) & \((6)\) & No \\
  1.2 & \(2\) & \(2\) & \(3\) & \(\mathbb{P}(1^{(3)}, 2)\) & \((4)\) & Non-strict \\
  1.3 & \(2\) & \(3\) & \(4\) & \(\mathbb{P}^3\) & \((3)\) & Strict \\
  \midrule
  2.1 & \(3\) & \(2\) & \(4\) & \(\mathbb{P}(1^{(4)}, 3)\) & \((6)\) & Strict \\
  2.2 & \(3\) & \(4\) & \(5\) & \(\mathbb{P}^4\) & \((4)\) & Strict \\
  2.3 & \(4\) & \(1\) & \(3\) & \(\mathbb{P}(1^{(3)}, 2^{(2)}, 3^{(2)})\) & \((6^{(2)})\) & No \\
  2.4 & \(4\) & \(2\) & \(4\) & \(\mathbb{P}(1^{(4)}, 2^{(2)}, 3)\) & \((4, 6)\) & No \\
  2.5 & \(4\) & \(4\) & \(5\) & \(\mathbb{P}(1^{(5)}, 2^{(2)})\) & \((4^{(2)})\) & Non-strict \\
  2.6 & \(4\) & \(6\) & \(6\) & \(\mathbb{P}(1^{(6)}, 2)\) & \((3, 4)\) & Non-strict \\
  2.7 & \(4\) & \(9\) & \(7\) & \(\mathbb{P}^6\) & \((3, 3)\) & Strict \\
  \midrule
  3.1 & \(4\) & \(1\) & \(4\) & \(\mathbb{P}(1^{(4)}, 2, 5)\) & \((10)\) & No \\
  3.2 & \(4\) & \(2\) & \(5\) & \(\mathbb{P}(1^{(5)}, 4)\) & \((8)\) & Strict \\
  3.3 & \(4\) & \(3\) & \(5\) & \(\mathbb{P}(1^{(5)}, 2)\) & \((6)\) & Non-strict \\
  3.4 & \(4\) & \(5\) & \(6\) & \(\mathbb{P}^5\) & \((5)\) & Strict \\
  3.5 & \(5\) & \(2\) & \(5\) & \(\mathbb{P}(1^{(5)}, 2, 3^{(2)})\) & \((6^{(2)})\) & No \\
  3.6 & \(5\) & \(4\) & \(6\) & \(\mathbb{P}(1^{(6)}, 2, 3)\) & \((4, 6)\) & Non-strict \\
  3.7 & \(5\) & \(8\) & \(7\) & \(\mathbb{P}(1^{(7)}, 2)\) & \((4^{(2)})\) & Non-strict \\
  3.8 & \(5\) & \(12\) & \(8\) & \(\mathbb{P}^7\) & \((3, 4)\) & Strict \\
  3.9 & \(6\) & \(1\) & \(4\) & \(\mathbb{P}(1^{(4)}, 2^{(3)}, 3^{(3)})\) & \((6^{(3)})\) & No \\
  3.10 & \(6\) & \(2\) & \(5\) & \(\mathbb{P}(1^{(5)}, 2^{(3)}, 3^{(2)})\) & \((4, 6^{(2)})\) & No \\
  3.11 & \(6\) & \(4\) & \(6\) & \(\mathbb{P}(1^{(6)}, 2^{(3)}, 3)\) & \((4^{(2)}, 6)\) & No \\
  3.12 & \(6\) & \(8\) & \(7\) & \(\mathbb{P}(1^{(7)}, 2^{(3)})\) & \((4^{(3)})\) & Non-strict \\
  3.13 & \(6\) & \(12\) & \(8\) & \(\mathbb{P}(1^{(8)}, 2^{(2)})\) & \((3, 4^{(2)})\) & Non-strict \\
  3.14 & \(6\) & \(18\) & \(9\) & \(\mathbb{P}(1^{(9)}, 2)\) & \((3^{(2)}, 4)\) & Non-strict \\
  3.15 & \(6\) & \(27\) & \(10\) & \(\mathbb{P}^9\) & \((3^{(3)})\) & Strict \\
  \midrule
  4.1 & \(5\) & \(2\) & \(6\) & \(\mathbb{P}(1^{(6)}, 5)\) & \((10)\) & Strict \\
  4.2 & \(5\) & \(6\) & \(7\) & \(\mathbb{P}^6\) & \((6)\) & Strict \\
  4.3 & \(6\) & \(1\) & \(5\) & \(\mathbb{P}(1^{(5)}, 2^{(2)}, 3, 5)\) & \((6, 10)\) & No \\
  4.4 & \(6\) & \(2\) & \(6\) & \(\mathbb{P}(1^{(6)}, 2, 3, 4)\) & \((6, 8)\) & No \\
  4.5 & \(6\) & \(2\) & \(6\) & \(\mathbb{P}(1^{(6)}, 2^{(2)}, 5)\) & \((4, 10)\) & No \\
  4.6 & \(6\) & \(3\) & \(6\) & \(\mathbb{P}(1^{(6)}, 2^{(2)}, 3)\) & \((6^{(2)})\) & No \\
  4.7 & \(6\) & \(3\) & \(7\) & \(\mathbb{P}(1^{(7)}, 2, 5)\) & \((3, 10)\) & No \\
  4.8 & \(6\) & \(4\) & \(7\) & \(\mathbb{P}(1^{(7)}, 3^{(2)})\) & \((6^{(2)})\) & Strict \\
  4.9 & \(6\) & \(5\) & \(7\) & \(\mathbb{P}(1^{(7)}, 2, 3)\) & \((5, 6)\) & Non-strict \\
  4.10 & \(6\) & \(6\) & \(7\) & \(\mathbb{P}(1^{(7)}, 2^{(2)})\) & \((4, 6)\) & Non-strict \\
  4.11 & \(6\) & \(6\) & \(8\) & \(\mathbb{P}(1^{(8)}, 4)\) & \((3, 8)\) & Strict \\
  4.12 & \(6\) & \(6\) & \(8\) & \(\mathbb{P}(1^{(8)}, 2)\) & \((3, 6)\) & Non-strict \\
  4.13 & \(6\) & \(8\) & \(8\) & \(\mathbb{P}(1^{(8)}, 3)\) & \((4, 6)\) & Strict \\
  4.14 & \(6\) & \(10\) & \(8\) & \(\mathbb{P}(1^{(8)}, 2)\) & \((4, 5)\) & Non-strict \\
  4.15 & \(6\) & \(15\) & \(8\) & \(\mathbb{P}^8\) & \((3, 5)\) & Strict \\
  4.16 & \(6\) & \(16\) & \(8\) & \(\mathbb{P}^8\) & \((4^{(2)})\) & Strict \\
  4.17 & \(7\) & \(2\) & \(6\) & \(\mathbb{P}(1^{(6)}, 2^{(2)}, 3^{(3)})\) & \((6^{(3)})\) & No \\
  4.18 & \(7\) & \(4\) & \(7\) & \(\mathbb{P}(1^{(7)}, 2^{(2)}, 3^{(2)})\) & \((4, 6^{(2)})\) & No \\
  4.19 & \(7\) & \(8\) & \(8\) & \(\mathbb{P}(1^{(8)}, 2^{(2)}, 3)\) & \((4^{(2)}, 6)\) & Non-strict \\
  4.20 & \(7\) & \(16\) & \(9\) & \(\mathbb{P}(1^{(9)}, 2^{(2)})\) & \((4^{(3)})\) & Non-strict \\
  4.21 & \(7\) & \(24\) & \(10\) & \(\mathbb{P}(1^{{(10)}}, 2)\) & \((3, 4^{(2)})\) & Non-strict \\
  4.22 & \(7\) & \(36\) & \(10\) & \(\mathbb{P}^{10}\) & \((3^{(2)}, 4)\) & Strict \\
  4.23 & \(8\) & \(1\) & \(5\) & \(\mathbb{P}(1^{(5)}, 2^{(4)}, 3^{(4)})\) & \((6^{(4)})\) & No \\
  4.24 & \(8\) & \(2\) & \(6\) & \(\mathbb{P}(1^{(6)}, 2^{(4)}, 3^{(3)})\) & \((4, 6^{(3)})\) & No \\
  4.25 & \(8\) & \(4\) & \(7\) & \(\mathbb{P}(1^{(7)}, 2^{(4)}, 3^{(2)})\) & \((4^{(2)}, 6^{(2)})\) & No \\
  4.26 & \(8\) & \(8\) & \(8\) & \(\mathbb{P}(1^{(8)}, 2^{(4)}, 3)\) & \((4^{(3)}, 6)\) & No \\
  4.27 & \(8\) & \(16\) & \(9\) & \(\mathbb{P}(1^{(9)}, 2^{(4)})\) & \((4^{(4)})\) & Non-strict \\
  4.28 & \(8\) & \(24\) & \(10\) & \(\mathbb{P}(1^{(10)}, 2^{(3)})\) & \((3, 4^{(3)})\) & Non-strict \\
  4.29 & \(8\) & \(36\) & \(11\) & \(\mathbb{P}(1^{(11)}, 2^{(2)})\) & \((3^{(2)}, 4^{(2)})\) & Non-strict \\
  4.30 & \(8\) & \(54\) & \(12\) & \(\mathbb{P}(1^{(12)}, 2)\) & \((3^{(3)}, 4)\) & Non-strict \\
  4.31 & \(8\) & \(81\) & \(13\) & \(\mathbb{P}^{12}\) & \((3^{(4)})\) & Strict  \\
  \bottomrule
  \caption{List of series and semi-series of smooth well-formed Fano weighted complete intersection of variance up to 4}
\end{longtable}

\clearpage
\printbibliography

\end{document}